\documentclass[a4paper,11pt]{amsart}
\usepackage[english]{babel}
\usepackage{amsmath,tikz-cd}
\usepackage{amssymb}
\usepackage[T1]{fontenc}    
\usepackage[utf8]{inputenc} 
\usepackage{lmodern}    
\usepackage{mathrsfs}
\usepackage{enumerate}
\usepackage{mathtools,yfonts}
\usepackage{tikz,tkz-euclide}
\usepackage{todonotes}
\usepackage{pgfplots}
\usetikzlibrary{arrows}
\usepackage{makecell}

\usepackage{xcolor} 
\definecolor{DarkRed}{RGB}{173,0,0}
\definecolor{LightRed}{RGB}{201,0,0}
\usepackage[
    colorlinks=true,
    linkcolor=DarkRed,
    urlcolor=LightRed,
    citecolor=LightRed
]{hyperref}

\usepackage{tcolorbox}
\usepackage[]{algorithm2e}

\newtheorem{thm}{Theorem}[section]
\newtheorem{Lemma}[thm]{Lemma}
\newtheorem{Proposition}[thm]{Proposition}
\newtheorem{Corollary}[thm]{Corollary}

\theoremstyle{definition}

\newtheorem{Definition}[thm]{Definition}
\newtheorem{Remark}[thm]{Remark}

\definecolor{wwwwww}{rgb}{0.4,0.4,0.4}

\newcommand{\PP}{\mathbb{P}}
\newcommand{\ZZ}{\mathbb{Z}}

\newcommand{\kk}{k}

\newcommand{\OO}{\mathcal{O}}   
\newcommand{\Tcal}{\mathcal{T}} 
\newcommand{\Ecal}{\mathcal{E}} 
\newcommand{\Ucal}{\mathcal{U}} 

\DeclareMathOperator{\Cox}{Cox} 
\DeclareMathOperator{\Spec}{Spec}
\DeclareMathOperator{\Pic}{Pic}
\DeclareMathOperator{\Br}{Br}
\DeclareMathOperator{\Aut}{Aut}

\renewcommand{\P}{\mathbb{P}}

\newcommand{\I}{\mathcal{I}}

\DeclareMathOperator{\Def}{Def}
\DeclareMathOperator{\Ext}{Ext}
\DeclareMathOperator{\End}{End}

\DeclareMathOperator{\Hom}{Hom}

\DeclareMathOperator{\Ima}{Im}

\DeclareMathOperator{\Sym}{Sym}

\hypersetup{pdfpagemode=UseNone}
\hypersetup{pdfstartview=FitH}

\newcommand{\TypeTwo}[3]{(#1,#2,#3)}
\newcommand{\Gcal}[3]{\mathcal{G}_{#1,#2,#3}} 
\newcommand{\VV}[6]{V^{\,1}_{#1,#2,#3,#4,#5,#6}}
\newcommand{\PVV}[6]{\mathbb{P}\!\big(\VV{#1}{#2}{#3}{#4}{#5}{#6}\big)}

\begin{document}

\title{On the unirationality of conic bundles with discriminant of degree eight}

\author[Alex Casarotti]{Alex Casarotti}
\address{\sc Alex Casarotti\\ Dipartimento di Matematica e Informatica, Universit\`a di Ferrara, Via Machiavelli 30, 44121 Ferrara, Italy}
\email{csrlxa@unife.it}

\author[Søren Gammelgaard]{Søren Gammelgaard}
\address{{\sc S{\o}ren Gammelgaard}\\ Dipartimento di Matematica e Informatica, Universit\`a di Ferrara, Via Machiavelli 30, 44121 Ferrara, Italy}
\email{soren.gammelgaard@unife.it}

\author[Alex Massarenti]{Alex Massarenti}
\address{\sc Alex Massarenti\\ Dipartimento di Matematica e Informatica, Universit\`a di Ferrara, Via Machiavelli 30, 44121 Ferrara, Italy}
\email{msslxa@unife.it}

\date{\today}

\begin{abstract}
We study the unirationality of surface conic bundles $\pi\colon S\to\PP^1$ over an arbitrary field $k$ with discriminant degree $d_S=8$, the first case beyond the del Pezzo range. We divide these surfaces in four families and produce explicit rational multisections via tangent constructions and Cremona transformations. Over $C_1$ fields we obtain Zariski dense loci of minimal, hence non $k$–rational, yet $k$–unirational conic bundles in each family; for one of the types we prove that the dense unirational locus is indeed Zariski open. Finally, we investigate the deformation theory of these conic bundles and how their unirationality behaves under specialization.
\end{abstract}

\maketitle
\setcounter{tocdepth}{1}
\tableofcontents

\section{Introduction}
The birational classification of projective varieties over a field intertwines geometry and arithmetic in a striking way, and conic bundles provide one of its most revealing test beds. A \emph{surface conic bundle} is a smooth, projective, geometrically irreducible surface \(S\) over a field \(k\) equipped with a morphism \(\pi\colon S\to\PP^1\) whose geometric fibers are plane conics (either smooth or the union of two lines). Over non closed fields, rationality is constrained by Brauer-theoretic obstructions, while unirationality can still occur abundantly; understanding the border between these behaviors has guided much of the modern progress.

A comprehensive survey of the rationality problem for conic bundles (and its higher-dimensional avatars) is given by Prokhorov~\cite{Pr18}. 

On the side of obstructions, recent advances for quadric and conic bundle fibrations show that subtle degeneration patterns often force stable non rationality; see, for instance, ~\cite{HKT16}, ~\cite{HPT18}, ~\cite{ABP18}, ~\cite{AO18}, ~\cite{BG18}, ~\cite{Sc19a} for quadric bundles. 

On the constructive side, two complementary mechanisms are particularly relevant to surfaces. First, the presence of a rational multisection of low degree forces unirationality by classical projection arguments \cite{KM17}, \cite{Mas23b}. Second, in special configurations one can diagonalize the generic fiber and reduce the existence of rational or unirational parametrizations to arithmetic statements about families of diagonal conics. 

A powerful example of the latter is due to Mestre~\cite{Mes94}, who showed that certain \(2\)-torsion Brauer classes with eight simple poles and constant residues can be annihilated by a change of variable, yielding explicit unirational parametrizations.

The case of del Pezzo surfaces carrying a conic bundle structure is now well understood. If \(d_S\) denotes the degree of the discriminant divisor of \(\pi\) (i.e.\ the number of singular fibers, counted with multiplicity), the adjunction formula gives
\[
K_S^2 \;=\; 8-d_S.
\]
Thus \(S\) is a del Pezzo surface if and only if \(d_S\le 7\). A decisive theorem of Kollár–Mella~\cite{KM17} shows that, over an arbitrary field \(k\), a conic bundle surface with \(d_S\le 7\) is unirational if and only if it has a \(k\)-point. In particular, del Pezzo surfaces of degree \(1\) and \(2\) admitting a conic bundle structure are unirational as soon as they possess a rational point. Note that this last requirement is automatically satisfied for del Pezzo surfaces of degree $1$, the rational point being the unique base point of $|-K_S|$. 

The case \(d_S=8\) is the first that falls outside the del Pezzo realm (\(K_S^2=0\)) and constitutes a watershed for the geometry and arithmetic of conic bundles.

\medskip

We set the framework and notation. Fix nonnegative integers \(a_0,a_1,a_2\) (which we will call \emph{weights}) and consider the projective bundle
\[
\Tcal_{a_0,a_1,a_2}\;\cong\;\PP\!\big(\Ecal_{a_0,a_1,a_2}\big)
\qquad\text{with}\qquad
\Ecal_{a_0,a_1,a_2}\;\cong\;\OO_{\PP^1}(a_0)\oplus  \OO_{\PP^1}(a_1)\oplus \OO_{\PP^1}(a_2).
\]
We view \(\Tcal_{a_0,a_1,a_2}\) as the simplicial toric variety with Cox ring
\(\Cox(\Tcal_{a_0,a_1,a_2})\cong k[x_0,x_1,y_0,y_1,y_2]\),
\(\ZZ^2\)–graded so that \(\deg x_0=\deg x_1=(1,0)\), \(\deg y_i=(-a_i,1)\), and the irrelevant ideal is \((x_0,x_1)\cap(y_0,y_1,y_2)\).
A conic bundle surface \(\pi\colon S\to\PP^1\) with a splitting presentation is cut out by a single quadratic equation in the \(\PP^2\)–fibers:
\[
S \;=\; \Big\{\,\textstyle\sum_{0\le i\le j\le 2}\sigma_{i,j}(x_0,x_1)\,y_i y_j=0\Big\}
\;\subset\;
\PP\!\big(\OO_{\PP^1}(a_0)\oplus \OO_{\PP^1}(a_1)\oplus \OO_{\PP^1}(a_2)\big),
\]
where each \(\sigma_{i,j}\in k[x_0,x_1]_{d_{i,j}}\) is homogeneous and the degrees satisfy the compatibilities
\[
d_{0,0}-2a_0 \;=\; d_{0,1}-a_0-a_1 \;=\; d_{0,2}-a_0-a_2 \;=\; d_{1,1}-2a_1 \;=\; d_{1,2}-a_1-a_2 \;=\; d_{2,2}-2a_2.
\]
The multidegree is the \(6\)–tuple \((d_{0,0},d_{0,1},d_{0,2},d_{1,1},d_{1,2},d_{2,2})\).
The discriminant of \(\pi\) is the divisor of points of \(\PP^1\) where the fiber conic is singular, and its degree is
\[
d_S \;=\; d_{0,0}+d_{1,1}+d_{2,2}.
\]
In the case \(d_S=8\) there are exactly four possible multidegrees:
\[
(4,3,3,2,2,2),\quad
(4,4,2,4,2,0),\quad
(6,4,3,2,1,0),\quad
(8,4,4,0,0,0),
\]
with corresponding weights:
\[
(2,1,1),\quad
(2,2,0),\quad
(3,1,0),\quad
(4,0,0),
\]
and each corresponding parameter space has dimension \(21\). {Given a conic bundle $S\to \P^1$, we term the tuple of weights $(a_0,a_1,a_2)$ the \emph{type} of $S$.}
Geometric constructions, involving tangent conditions or special linear subfamilies, are an efficient way to produce unirational multisections ~\cite{MM24,Mas23b}.

\medskip

Our results focus on the watershed case $d_S=8$. Working uniformly in the four $21$–dimensional families listed above, we construct explicit unirational multisections and show that unirationality persists on large loci of the parameter spaces. We also separate rational from non-rational behavior inside the same family via Brauer residues. 

\begin{thm} Let $k$ be an {infinite} perfect $C_1$–field of characteristic different from 2, and $\mathbb{P}_{(d_{0,0},d_{0,1},d_{0,2},d_{1,1},d_{1,2},d_{2,2})}$ be the parameter space for conic bundles of multidegree $(d_{0,0},d_{0,1},d_{0,2},d_{1,1},d_{1,2},d_{2,2})$. For each
$$
(d_{0,0},d_{0,1},d_{0,2},d_{1,1},d_{1,2},d_{2,2}) \in \{(4,3,3,2,2,2), (4,4,2,4,2,0), (6,4,3,2,1,0),(8,4,4,0,0,0)\},
$$
there exists a Zariski dense subset of $\mathbb{P}_{(d_{0,0},d_{0,1},d_{0,2},d_{1,1},d_{1,2},d_{2,2})}$ whose general point represents a minimal, in particular non $k$–rational, yet $k$–unirational conic bundle surface $S\to\PP^1$ with $d_S=8$. Furthermore, in $\mathbb{P}_{(6,4,3,2,1,0)}$ the unirational locus is Zariski open. 
\end{thm}

We also analyze rational vs non rational behavior in the family $\mathbb{P}_{(8,4,4,0,0,0)}$ and show that both phenomena occur on Zariski dense subsets: we construct a dense locus of rational conic bundles and another dense locus of non rational conic bundles, the latter detected by a nontrivial $2$–torsion Brauer class.

\medskip

For {the family of conic bundles of multidegree} \((4,3,3,2,2,2)\), unirationality on a dense locus is established via a rational \(2\)–section cut out by a bitangent hyperplane section; see Proposition~\ref{prop:TTUni433222} and Theorem~\ref{thm:main433222}. Minimality is addressed using Brauer residues; see Proposition~\ref{prop:nosec433222} and Proposition~\ref{prop:min433222}.

For {the family of multidegree} \((4,4,2,4,2,0)\), unirationality on a dense locus is proved by producing a rational \(2\)–section through prescribed tangent conditions, together with dominance of the automorphism action on the parameter space; see Proposition~\ref{prop:TTUni442420} and Proposition~\ref{prop:dom442420}. Density of rational points in suitable families is discussed in Proposition~\ref{prop:elld442420}.

For {the family of multidegree} \((6,4,3,2,1,0)\), a general member is unirational; moreover, the unirational locus is Zariski open; see Proposition~\ref{prop:case643210}.

For {the family of multidegree} \((8,4,4,0,0,0)\), we exhibit a dense locus of unirational conic bundles using a diagonal normal. We further construct a dense locus of rational bundles (Proposition~\ref{prop:c2zero}), a transverse dense locus with unirational but non rational minimal conic bundles (Proposition~\ref{prop:min844}), and a dense locus of non rational bundles (Proposition~\ref{prop:844NotRational}). These results together are summarized in Theorem~\ref{thm:main844}.

\medskip

In Section~\ref{defs} we complement these constructions with a deformation--theoretic analysis of the
four families. We first study infinitesimal deformations of a splitting conic bundle
\(\pi\colon S\to\PP^{1}\) inside its ambient projective bundle
\(\PP(\Ecal_{a_0,a_1,a_2})\), proving an exact sequence that relates
embedded deformations of the pair \((S\subset\PP(\Ecal_{a_0,a_1,a_2}))\) to
\(H^{1}(\PP^{1},\End \Ecal_{a_0,a_1,a_2})\) and to the cohomology of the normal bundle
\(N_{S/\PP(\Ecal_{a_0,a_1,a_2})}\). We show that for
discriminant degree \(d_{S}=8\) the bundles of type \((2,1,1)\) are rigid, whereas in the three
other types \((2,2,0)\), \((3,1,0)\), and \((4,0,0)\) there exist embedded first--order
deformations that necessarily change the splitting type of the ambient rank–\(3\) bundle.
We then construct explicit one–parameter families realizing these deformations, degenerating
conic bundles of type \((2,1,1)\) to type \((2,2,0)\), of type \((2,2,0)\) to type \((3,1,0)\),
and of type \((3,1,0)\) to type \((4,0,0)\). In particular, the unirational conic bundles
produced in Sections~\ref{sec:433222}, \ref{sec:442420}, \ref{sec:643210}, \ref{sec:844000} sit in positive–dimensional families along which the splitting type
varies, and our degeneration arguments allow us to transport unirationality across the four
multidegree types with \(d_{S}=8\).

\subsection*{Conventions on the base field and terminology} 
All along the paper we work over a base field $k$ with characteristic different from two. {We shall always assume $k$ to be infinite.}

Let $X$ be a variety over $k$. When we say that $X$ is rational or unirational, without specifying over which field, we will always mean that $X$ is rational or unirational over $k$. Similarly, we will say that $X$ has a point or contains a variety with certain properties meaning that $X$ has a $k$-rational point or contains a variety defined over $k$ with the required properties. 

\subsection*{Acknowledgments}
The authors were supported by the PRIN 2022 grant (project 20223B5S8L) “\emph{Birational geometry of moduli spaces and special varieties}”. The third–named author is a member of GNSAGA (INdAM). We are grateful to Alena Pirutka for many helpful discussions on Brauer groups. We thank J\'anos Koll\'ar for directing our attention to the deformation questions in Section \ref{defs} and for helpful suggestions, and Barbara Fantechi for useful discussion on embedded infinitesimal deformations of conic bundles.

\section{Conic bundles, Enriques' criterion and Brauer groups}\label{QB}
In this section we fix notation for conic bundles and recall Enriques' unirationality criterion.

\begin{Definition}\label{def1}
Let \(W\) be a smooth \((n-1)\)-dimensional variety over a field \(k\), let \(\mathcal{E}\) be a rank \(3\) vector bundle on \(W\), let \(\overline{\pi}\colon \mathbb{P}(\mathcal{E})\to W\) be the associated projective bundle with tautological line bundle \(\mathcal{O}_{\mathbb{P}(\mathcal{E})}(1)\), and let \(\mathcal{L}\) be a line bundle on \(W\).
A quadratic form with values in \(\mathcal{L}\) is a global section
\[
\sigma \in H^0\!\big(W,\Sym^2\mathcal{E}^{\vee}\!\otimes \mathcal{L}\big)
\cong H^0\!\big(\mathbb{P}(\mathcal{E}), \mathcal{O}_{\mathbb{P}(\mathcal{E})}(2)\!\otimes \overline{\pi}^{*}\mathcal{L}\big).
\]
A conic bundle over \(W\) is a variety
\[
\mathcal{Q}^{1} \coloneqq \{\sigma=0\}\subset \mathbb{P}(\mathcal{E})
\]
cut out by a generically non degenerate section \(\sigma\) as above, endowed with the projection \(\pi\colon \mathcal{Q}^{1}\to W\) induced by \(\overline{\pi}\).

The discriminant of \(\pi\) is the divisor \(D_{\mathcal{Q}^{1}}\subset W\) where \(\sigma\) fails to have full rank. When \(W=\mathbb{P}^{n-1}\), the discriminant is a hypersurface, and we write \(d_{\mathcal{Q}^{1}}\coloneqq\deg(D_{\mathcal{Q}^{1}})\).
\end{Definition}

\begin{Remark}
A general fiber of \(\pi\colon \mathcal{Q}^{1}\to W\) is a smooth conic in \(\mathbb{P}^{2}\).
In Definition~\ref{def1} we do not require \(\pi\) to be flat. Often one defines a conic bundle as a flat morphism \(\pi\colon \mathcal{Q}^{1}\to W\) with fibers isomorphic to plane conics and with \(\mathcal{Q}^{1}\) smooth. In that case there exist a rank \(3\) bundle \(\mathcal{E}\to W\), a line bundle \(\mathcal{L}\to W\), and a section
\(\sigma\in H^0\!\big(\mathbb{P}(\mathcal{E}), \mathcal{O}_{\mathbb{P}(\mathcal{E})}(2)\otimes \overline{\pi}^{*}\mathcal{L}\big)\)
such that \(\mathcal{Q}^{1}=\{\sigma=0\}\subset \mathbb{P}(\mathcal{E})\) and \(\pi=\overline{\pi}\vert_{\mathcal{Q}^{1}}\)
\cite[Theorem~3.2]{Pr18}, \cite[Proposition~1.2]{Bea77}.
\end{Remark}

\begin{Definition}\label{ext_min}
A conic bundle \(\pi\colon \mathcal{Q}^{1}\to W\) is \emph{extremal} if its relative Picard rank is \(1\). It is \emph{minimal} if, in addition, it admits no rational section.
\end{Definition}

Let \(a_0,a_1,a_2\in \mathbb{Z}_{\ge 0}\). Consider the simplicial toric variety
\(\mathcal{T}_{a_0,a_1,a_2}\) with Cox ring
\[
\Cox(\mathcal{T}_{a_0,a_1,a_2})\cong k[x_0,\dots,x_{n-1},\,y_0,y_1,y_2]
\]
endowed with a \(\mathbb{Z}^2\)-grading, with respect to a chosen basis \((H_1,H_2)\) of \(\Pic(\mathcal{T}_{a_0,a_1,a_2})\), given by
\[
\left(
\begin{array}{cccccc}
x_0 & \dots & x_{n-1} & y_0 & y_1 & y_2\\\hline
1   & \dots & 1       & -a_0 & -a_1 & -a_2\\
0   & \dots & 0       & 1    & 1    & 1
\end{array}\right),
\qquad
\text{with irrelevant ideal }(x_0,\dots,x_{n-1})\cap (y_0,y_1,y_2).
\]
Then
\[
\mathcal{T}_{a_0,a_1,a_2}\cong \mathbb{P}(\mathcal{E}_{a_0,a_1,a_2}),\qquad
\mathcal{E}_{a_0,a_1,a_2}\cong \mathcal{O}_{\mathbb{P}^{n-1}}(a_0)\oplus \mathcal{O}_{\mathbb{P}^{n-1}}(a_1)\oplus \mathcal{O}_{\mathbb{P}^{n-1}}(a_2).
\]
In the secondary fan, \(H_1=(1,0)\) corresponds to the \(x_i\), \(H_2=(0,1)\), and \(v_i=(-a_i,1)\) corresponds to \(y_i\) for \(i=0,1,2\).

\begin{Definition}\label{def:splittoric}
A divisor \(D\subset \mathcal{T}_{a_0,a_1,a_2}\) of multidegree \((d_{0,0},\,d_{0,1},\,d_{0,2},\,d_{1,1},\,d_{1,2},\,d_{2,2};\,2)\)
is a hypersurface given by an equation
\[
\sum_{i\le j}\sigma_{i,j}(x_0,\dots,x_{n-1})\, y_i y_j \;=\; 0,
\]
where \(\sigma_{i,j}\in k[x_0,\dots,x_{n-1}]_{d_{i,j}}\) and the degrees satisfy the compatibility relations
\stepcounter{thm}
\begin{equation}\label{compdeg_conic}
d_{0,0}-2a_0
=d_{0,1}-a_0-a_1
=d_{0,2}-a_0-a_2
=d_{1,1}-2a_1
=d_{1,2}-a_1-a_2
=d_{2,2}-2a_2.
\end{equation}
\end{Definition}

Without loss of generality we assume \(a_0\ge a_1\ge a_2\), so \eqref{compdeg_conic} implies
\(d_{0,0}\ge d_{0,1}\ge d_{0,2}\ge d_{1,1}\ge d_{1,2}\ge d_{2,2}\).

\subsection{Brauer groups and residues}\label{Bra}
Assume for this subsection that \(k\) is a perfect field of characteristic different from two. The Brauer group \(\Br(k)\) is the group of equivalence classes of central simple \(k\)-algebras (CSAs), where
$$
A\sim B \text{ if } A\otimes_k M_m(k)\cong B\otimes_k M_n(k) \text{ for some } m,n\ge 1.
$$
Given \(a,b\in k^{\times}\), the quaternion algebra \((a,b)_k\) is the \(k\)-algebra generated by \(i,j\) with
$$
i^2=a,\: j^2=b,\: ij=-ji.
$$
It is a \(4\)-dimensional CSA over \(k\).

Let \(W\) be an integral variety over \(k\) with function field \(k(W)\).
For \(a,b\in k(W)^{\times}\), write \((a,b)\in \Br(k(W))[2]\) for the Brauer class of the conic
\(C_{(a,b)}=\{a x^2+b y^2 - z^2=0\}\) over \(k(W)\).

Let \(D\subset W\) be a prime divisor. Assume that \(W\) is smooth at the generic point of \(D\) and let \(v_D\) be the associated discrete divisorial valuation. Consider the residue maps
\[
\partial^1_D \colon k(W)^{\times}/k(W)^{\times 2}\longrightarrow \mathbb{Z}/2,\qquad
\partial^1_D(a)\equiv v_D(a)\pmod 2,
\]
and
\[
\partial^2_D \colon \Br(k(W))[2]\longrightarrow k(D)^{\times}/k(D)^{\times 2},\qquad
\partial^2_D(a,b) = (-1)^{v_D(a)v_D(b)}\,\overline{\frac{a^{v_D(b)}}{b^{v_D(a)}}},
\]
where the bar denotes the image in the residue field \(k(D)\).

For \(W=\mathbb{P}^1\), the Faddeev exact sequence yields
\[
0\longrightarrow \Br(k)[2]\longrightarrow \Br(k(t))[2]\longrightarrow \bigoplus_{p\in \mathbb{P}^1} k(p)^{\times}/k(p)^{\times 2},
\]
where the component at \(p\) is \(\partial^2_p\).
Hence, if \((a(t),b(t))=0\) in \(\Br(k(t))\), then all its residues \(\partial^2_p(a(t),b(t))\) vanish.
A conic bundle
\(\mathcal{Q}^1=\{a(t)x^2+b(t)y^2-z^2=0\}\to \mathbb{P}^1\) has a rational section if and only if the Brauer class \((a(t),b(t))\) is zero in \(\Br(k(t))\).
To prove non triviality of \((a(t),b(t))\), it suffices to find a point \(p\) with \(\partial^2_p(a(t),b(t))\ne 0\).
For background on Brauer groups and residues see \cite{CS21}.

\begin{Lemma}\label{hyp}
Let \(X_d\subset \mathbb{P}^{n+1}\) be a hypersurface of degree \(d\) having multiplicity \(m\) along an \(h\)-plane \(\Lambda\), and let \(\widetilde{X}_d\) be the blow-up of \(X_d\) along \(\Lambda\) with exceptional divisor \(\widetilde{E}\subset \widetilde{X}_d\).
Then \(\widetilde{X}_d\) is isomorphic to a divisor of multidegree
\[
(d,\,\underbrace{d-1,\dots,d-1}_{\binom{h+1}{1}},\,\dots,\,\underbrace{d-j,\dots,d-j}_{\binom{h+j}{j}},\,\dots,\,\underbrace{m,\dots,m}_{\binom{h+d-m}{d-m}};\; d-m)
\]
in \(\mathcal{T}_{1,0,\dots,0}\). Moreover \(\widetilde{E}\) is a divisor of bidegree \((m,d-m)\) in \(\mathbb{P}^{n-h}\times \mathbb{P}^{h}\).
\end{Lemma}

\begin{proof}
We may assume \(\Lambda=\{z_0=\cdots=z_{n-h}=0\}\) and
\[
X_d=\Big\{\;\sum_{m_0+\cdots+m_{n-h}=m} z_0^{m_0}\cdots z_{n-h}^{m_{n-h}}\,
A_{m_0,\dots,m_{n-h}}(z_0,\dots,z_{n+1})=0\;\Big\}\subset \mathbb{P}^{n+1},
\]
with \(A_{m_0,\dots,m_{n-h}}\in k[z_0,\dots,z_{n+1}]_{d-m}\).
The blow-up of \(\mathbb{P}^{n+1}\) along \(\Lambda\) is the simplicial toric variety \(\mathcal{T}\) with Cox ring
\(\Cox(\mathcal{T})\cong k[x_0,\dots,x_{n-h},y_0,\dots,y_{h+1}]\) and \(\mathbb{Z}^2\)-grading
\[
\left(
\begin{array}{cccccccc}
x_0 & \dots & x_{n-h} & y_0 & y_1 & \dots & y_{h+1}\\\hline
1   & \dots & 1       & 0   & 1   & \dots & 1\\
-1  & \dots & -1      & 1   & 0   & \dots & 0
\end{array}\right).
\]
After elementary row operations this becomes
\stepcounter{thm}
\begin{equation}\label{gM2}
\left(
\begin{array}{cccccccc}
x_0 & \dots & x_{n-h} & y_0 & y_1 & \dots & y_{h+1}\\\hline
1   & \dots & 1       & -1  & 0   & \dots & 0\\
0   & \dots & 0       & 1   & 1   & \dots & 1
\end{array}\right),
\end{equation}
so \(\mathcal{T}=\mathcal{T}_{1,0,\dots,0}\).
The blow-down is
\[
\phi\colon \mathcal{T}_{1,0,\dots,0}\to \mathbb{P}^{n+1},\qquad
(x_0,\dots,x_{n-h},y_0,\dots,y_{h+1})\mapsto [x_0y_0:\cdots:x_{n-h}y_0:y_1:\dots:y_{h+1}],
\]
and the exceptional divisor is \(E=\{y_0=0\}\).
Thus the strict transform is
$$
\widetilde{X}_d=\Big\{\; \sum_{m_0+\cdots+m_{n-h}=m}
x_0^{m_0}\cdots x_{n-h}^{m_{n-h}}\,
A_{m_0,\dots,m_{n-h}}(x_0y_0,\dots,x_{n-h}y_0,y_1,\dots,y_{h+1})=0\;\Big\}\subset \mathcal{T}_{1,0,\dots,0},
$$
which yields the asserted multidegree. From \eqref{gM2} we have
\(E\cong \mathbb{P}^{n-h}_{(x_0,\dots,x_{n-h})}\times \mathbb{P}^h_{(y_1,\dots,y_{h+1})}\),
and \(\widetilde{E}=\widetilde{X}_d\cap E\) has bidegree \((m,d-m)\).
\end{proof}

\begin{Definition}\label{def:splitting}
A \emph{splitting conic bundle} \(\pi\colon \mathcal{Q}^1\to \mathbb{P}^{n-1}\) is given by an equation
\stepcounter{thm}
\begin{equation}\label{Cox_g}
\mathcal{Q}^1=\Big\{\; \sum_{0\le i\le j\le 2} \sigma_{i,j}(x_0,\dots,x_{n-1})\, y_i y_j \;=\; 0\;\Big\} \subset
\mathbb{P}\!\big(\mathcal{O}_{\mathbb{P}^{n-1}}(a_0)\oplus \mathcal{O}_{\mathbb{P}^{n-1}}(a_1)\oplus \mathcal{O}_{\mathbb{P}^{n-1}}(a_2)\big),
\end{equation}
where \(\sigma_{i,j}\) is homogeneous of degree \(d_{i,j}\), and the degrees satisfy
\stepcounter{thm}
\begin{equation}\label{compdeg}
d_{0,0}-2a_0=d_{0,1}-a_0-a_1=d_{0,2}-a_0-a_2=d_{1,1}-2a_1=d_{1,2}-a_1-a_2=d_{2,2}-2a_2.
\end{equation}
The multidegree is \((d_{0,0},d_{0,1},d_{0,2},d_{1,1},d_{1,2},d_{2,2})\in \mathbb{Z}^6\).
\end{Definition}

For a splitting conic bundle of multidegree \((d_{0,0},\dots,d_{2,2})\) the discriminant has degree
\(d_{\mathcal{Q}^1}=d_{0,0}+d_{1,1}+d_{2,2}\).
Assuming \(a_0\ge a_1\ge a_2\), \eqref{compdeg} gives \(d_{0,0}\ge d_{1,1}\ge d_{2,2}\), and the nonzero diagonal degrees \(d_{i,i}\) have the same parity.

Seeing as $\P(\mathcal E) = \P(\mathcal E(i))$ for any vector bundle $\mathcal E$ and any twist $i$, we can choose our $a_0,a_1,a_2$ such that the equation \eqref{compdeg} equals 0, in other words $a_0 = d_{0,0}/2,\ a_1= d_{1,1}/2,\ a_2 = d_{2,2}/2$.

Let \(N(n-1,d)\coloneqq\binom{d+n-1}{\,n-1\,}\). Splitting conic bundles of multidegree
\((d_{0,0},\dots,d_{2,2})\) over \(\mathbb{P}^{n-1}\) correspond to points of
\[
V^{\,n-1}_{d_{0,0},\dots,d_{2,2}}=
\bigoplus_{0\le i\le j\le 2} k^{\,N(n-1,d_{i,j})},
\]
up to rescaling. We call \(\mathcal{Q}^1\) general if it corresponds to a general point of
\(\mathbb{P}\!\big(V^{\,n-1}_{d_{0,0},\dots,d_{2,2}}\big)\).

\begin{Remark}\label{Gro}
By the Birkhoff–Grothendieck splitting theorem \cite[Theorem~4.1]{HM82}, every vector bundle on \(\mathbb{P}^1\) splits as a sum of line bundles; in particular, every conic bundle \(\mathcal{Q}^1\to \mathbb{P}^1\) is splitting.
\end{Remark}

\begin{Remark}[Lang’s theorem]\label{lang}
Fix \(r\in \mathbb{R}_{\ge 0}\). A field \(k\) is \(C_r\) if and only if every homogeneous form \(f\in k[x_0,\dots,x_n]_d\) with \(d>0\) in \(n+1\) variables has a nontrivial zero whenever \(n+1>d^{\,r}\).
If \(k\) is \(C_r\) and \(f_1,\dots,f_s\in k[x_0,\dots,x_n]_d\) are of the same degree with \(n+1> s\, d^{\,r}\), then they have a nontrivial common zero \cite[Prop.~1.2.6]{Poo17}; moreover \(k(t)\) is \(C_{r+1}\) \cite[Theorem.~1.2.7]{Poo17}.
Let \(\pi\colon \mathcal{Q}^1\to \mathbb{P}^{n-1}\) be a conic bundle over a \(C_r\) field \(k\) with generic fiber \(\mathcal{Q}^1_{\eta}\) over \(F=k(t_1,\dots,t_{n-1})\).
If \(\mathcal{Q}^1_{\eta}(F)\neq \varnothing\) then the projection from an \(F\)-point shows that \(\mathcal{Q}^1_{\eta}\) is \(F\)-rational, hence \(\mathcal{Q}^1\) is \(k\)-rational.
\end{Remark}
The $C_1$ property for a field is also known as being \emph{quasi-algebraically closed}.

We will also use the following formulation of Enriques’ unirationality criterion \cite[Proposition~10.1.1]{IP99}.

\begin{Proposition}\label{Enr}
Let \(\pi\colon \mathcal{Q}^1\to W\) be a conic bundle over a unirational variety \(W\).
Then \(\mathcal{Q}^1\) is unirational if and only if there exists a unirational subvariety \(Z\subset \mathcal{Q}^1\) such that \(\pi_{|Z}\colon Z\to W\) is dominant.
\end{Proposition}

\begin{proof}
If \(\mathcal{Q}^1\) is unirational, there is a dominant rational map \(\psi\colon \mathbb{P}^N \dashrightarrow \mathcal{Q}^1\).
For a general \((N-1)\)-plane \(H\subset \mathbb{P}^N\), the closure \(Z=\overline{\psi(H)}\) is unirational and transverse to the fibration, hence \(\pi_{|Z}\) is dominant.

Conversely, assume \(Z\subset \mathcal{Q}^1\) is unirational and \(\pi_{|Z}\) is dominant. Consider the fiber product
\[
\begin{tikzpicture}[xscale=2.6,yscale=-1.3]
\node (A0_0) at (0, 0) {$\mathcal{Q}^1\times_{W} Z$};
\node (B)     at (1, 0) {$\mathcal{Q}^1$};
\node (A1_0)  at (0, 1) {$Z$};
\node (A1_1)  at (1, 1) {$W$};
\path (A0_0) edge [->] (B);
\path (B)    edge [->] node [auto] {$\scriptstyle \pi$} (A1_1);
\path (A1_0) edge [->] node [auto] {$\scriptstyle \pi_{|Z}$} (A1_1);
\path (A0_0) edge [->] (A1_0);
\end{tikzpicture}
\]
The projection \(\mathcal{Q}^1\times_W Z\to Z\) is a conic bundle admitting a rational section \(Z\dashrightarrow \mathcal{Q}^1\times_W Z\).
Such a section yields a \(k(Z)\)-point on the generic fiber, and projection from this point gives a birational isomorphism
\(\mathcal{Q}^1\times_W Z \sim_{\mathrm{bir}} Z\times \mathbb{P}^1\).
Since \(Z\) is unirational, so is \(\mathcal{Q}^1\times_W Z\), and hence \(\mathcal{Q}^1\) is unirational.
\end{proof}

\begin{Remark}\label{not0}
For a splitting conic bundle \eqref{Cox_g}, if one of the diagonal coefficients vanishes identically, say \(\sigma_{0,0}\equiv 0\), then \(\{y_1=y_2=0\}\) gives a rational section of \(\pi\colon \mathcal{Q}^1\to \mathbb{P}^{n-1}\), so \(\mathcal{Q}^1\) is rational \cite[Definition~21]{Sc19a}.
More generally, if \(d_{i,i}<0\) then \(\sigma_{i,i}=0\) and we are in the same situation. Throughout we therefore assume \(d_{i,i}\ge 0\) and \(\sigma_{i,i}\not\equiv 0\) for \(i=0,1,2\).
\end{Remark}

\begin{Remark}\label{topsiR}
For a surface conic bundle (\(n=2\)) one has \(({-}K_{\mathcal{Q}^1})^2=8-d_{\mathcal{Q}^1}\) \cite[p.~5]{Ko17}.
In particular, \(d_{\mathcal{Q}^1}=8\) is the minimal number of singular fibers for which a conic bundle surface is not del Pezzo.
\end{Remark}

\begin{Proposition}\label{diag}
Let \(k\) be a field with \(\operatorname{char} k\ne 2\), and let \(\alpha_0,\dots,\alpha_5\in k\) with \(\alpha_0\ne 0\) and \(\alpha_1^2-4\alpha_0\alpha_3\ne 0\).
The quadratic form
\[
\mathcal{Q}=\alpha_0 y_0^2+\alpha_1 y_0y_1+\alpha_2 y_0y_2+\alpha_3 y_1^2+\alpha_4 y_1y_2+\alpha_5 y_2^2
\]
in \(y_0,y_1,y_2\) is congruent over \(k\) to
\[
\mathcal{Q}'=\alpha_0 z_0^2+\alpha_0(4\alpha_0\alpha_3-\alpha_1^2) z_1^2
+4(4\alpha_0\alpha_3-\alpha_1^2)\,\delta_{\mathcal{Q}}\, z_2^2,
\]
where
\[
\delta_{\mathcal{Q}}=\alpha_0\alpha_3\alpha_5-\tfrac{1}{4}\alpha_0\alpha_4^2
-\tfrac{1}{4}\alpha_1^2\alpha_5+\tfrac{1}{4}\alpha_1\alpha_2\alpha_4-\tfrac{1}{4}\alpha_2^2\alpha_3
\]
is the discriminant of \(\mathcal{Q}\).
\end{Proposition}

\begin{proof}
The vectors
\(e_1=(1,0,0)\), \(e_2=(\alpha_1,2\alpha_0,0)\), and
\(e_3=(\alpha_1\alpha_4-2\alpha_2\alpha_3,\;\alpha_1\alpha_2-2\alpha_0\alpha_4,\;-\alpha_1^2+4\alpha_0\alpha_3)\)
form a basis orthogonal with respect to \(\mathcal{Q}\), and \(\mathcal{Q}\) takes the stated diagonal form in this basis.
\end{proof}


\begin{Lemma}\label{sect}
	Let \( \left(\mathcal{Q}_u, s_u\right)\) be a family of conic bundles with a section, with \[\mathcal{Q}_u =\{F_u(x_0,y_0,y_1)=0\}\subset \mathbb{A}^3,\] parametrized by $u\in k^\times$, such that the family $\left(\mathcal{Q}_u \right)$  degenerates to a smooth conic bundle \(\mathcal{Q}_0=\{F_0(x_0,y_0,y_1)=0\}\).
	Then  \(\mathcal{Q}_0\) also has a section.
\end{Lemma}

\begin{proof}
Let
\[
s_u\colon \mathbb{A}^1\to \mathcal{Q}_u,\qquad t\mapsto (t,a_u(t),b_u(t)),
\]
be the family of sections. Set $(a_0, b_0) \coloneqq \lim_{u\to 0} (a_u, b_u)$, and set \(s_0=\lim_{u\to 0}s_u=(t,a_0(t),b_0(t))\).
The collapse of \(s_u\) to a point would produce a singularity on \(\mathcal{Q}_0\), which is excluded.
Since \(F_u(t,a_u(t),b_u(t))=0\) for all \(u\ne 0\) and \(t\), we get
\[
F_0(t,a_0(t),b_0(t))=\lim_{u\to 0} F_u(t,a_u(t),b_u(t))=0,
\]
so \(s_0\) is a section of \(\mathcal{Q}_0\).
\end{proof}

From now on, we restrict ourselves to considering conic bundles $\mathcal{Q}^1$ over $\mathbb{P}^1$ with eight singular fibers. Then $\mathcal{Q}^1\subset \P(\mathcal{E}_{a_0,a_1,a_2})$ for which the triple $(a_0,a_1,a_2)$ is one of $(2,1,1),\, (2,2,0),\, (3,1,0),\, $ or $(4,0,0)$. {As mentioned in the Introduction,} given $\mathcal Q^1$, we call the corresponding triple $(a_0,a_1,a_2)$ its type. We shall now analyze the conic bundles of each type separately.

\begin{Remark}
	As can be seen by writing out explicit equations, if a conic bundle $\mathcal Q^1$ belongs to more than one type, in its equation \eqref{Cox_g}, at least one of the $\sigma_{i,i}$ must vanish. In this case $\mathcal Q^1$ is always rational by Remark \ref{not0}, and we shall not consider such bundles further.
\end{Remark}

\section{Conic bundles of type \texorpdfstring{$\TypeTwo{2}{1}{1}$}{(2,1,1)}}\label{sec:433222}
We start by considering conic bundles of the form
\[
\mathcal{Q}^1
=\{\sigma_{0,0}y_0^2+\sigma_{0,1}y_0y_1+\sigma_{0,2}y_0y_2+\sigma_{1,1}y_1^2+\sigma_{1,2}y_1y_2+\sigma_{2,2}y_2^2=0\}
\subset \Tcal_{2,1,1},
\]
with $\sigma_{0,0}\in k[x_0,x_1]_4$, $\sigma_{0,1},\sigma_{0,2}\in k[x_0,x_1]_3$, and $\sigma_{1,1},\sigma_{1,2},\sigma_{2,2}\in k[x_0,x_1]_2$. By Lemma~\ref{hyp}, the image of $\mathcal{Q}^1$ via
\[
\begin{array}{cccc}
 \phi: & \Tcal_{2,1,1} & \longrightarrow & \mathbb{P}^{3}\\
  & (x_0,x_{1},y_0,y_1,y_{2}) & \mapsto & [\,x_0y_0:x_{1}y_0:y_1:y_{2}\,]
\end{array}
\]
is a quartic surface $X_4\subset\mathbb{P}^3$ with double points along the line $\Lambda\coloneqq\{z_0=z_1=0\}$, and $\mathcal{Q}^1$ is the blow-up of $X_4$ along $\Lambda$. Write
\stepcounter{thm}
\begin{equation}\label{eq:pol433222}
\begin{aligned}
\sigma_{0,0} &= a_0x_0^4 + a_1x_0^3x_1 + a_2x_0^2x_1^2 + a_3x_0x_1^3 + a_4x_1^4,\\
\sigma_{0,1} &= b_0x_0^3 + b_1x_0^2x_1 + b_2x_0x_1^2 + b_3x_1^3,\\
\sigma_{0,2} &= c_0x_0^3 + c_1x_0^2x_1 + c_2x_0x_1^2 + c_3x_1^3,\\
\sigma_{1,1} &= d_0x_0^2 + d_1x_0x_1 + d_2x_1^2,\\
\sigma_{1,2} &= g_0x_0^2 + g_1x_0x_1 + g_2x_1^2,\\
\sigma_{2,2} &= h_0x_0^2 + h_1x_0x_1 + h_2x_1^2.
\end{aligned}
\end{equation}

\begin{Proposition}\label{prop:TTUni433222}
Assume the base field $k$ is $C_1$. If there exist two points $p,q\in \mathcal{Q}^1\setminus\{y_0=0\}$ on distinct fibers of the conic bundle such that
\(
T_{\phi(p)}X_4 = T_{\phi(q)}X_4,
\)
and $\mathcal{Q}^1$ is otherwise general, then $\mathcal{Q}^1$ is unirational.
\end{Proposition}

\begin{proof}
Set $p'=\phi(p)$ and $q'=\phi(q)$; note $p',q'\notin\Lambda$. Then $T_{p'}X_4=T_{q'}X_4$, and
\(
C_4\coloneqq X_4\cap T_{p'}X_4
\)
is a quartic plane curve singular at $p',q'$, and at $x\coloneqq T_{p'}X_4\cap\Lambda$. For a plane $H\subset\mathbb{P}^3$ containing $\Lambda$, the residual intersection $H\cap X_4=2\Lambda + C_H$ is a conic $C_H$, and $H\cap T_{p'}X_4$ meets $C_H$ in two points lying on $C_4$, so $C_4$ is a $2$-section of $\mathcal{Q}^1\to\mathbb{P}^1$.

If $C_4$ has a triple point, it splits as the union of two sections, hence $\mathcal{Q}^1$ is rational. Otherwise $C_4$ has double points at $p',q',x$. The quadratic Cremona transformation centered at $p',q',x$ maps $C_4$ birationally to a conic $C_2$. Over a $C_1$ field, $C_2(k)\neq\varnothing$, so $C_2$ (hence $C_4$) is rational. Thus $C_4$ is a rational $2$-section and Proposition~\ref{Enr} yields the claim.
\end{proof}

\begin{Remark}\label{rmk:example433222}
The proof works over any field provided $C_2$ has a rational point. For instance, in \eqref{eq:pol433222} take
\[
\begin{array}{l}
a_0=a_1=a_3=a_4=0; \\ 
a_2=b_0=b_2=b_3=c_0=c_1=c_3=d_0=d_1=d_2=g_1=h_0=h_1=1; \\ 
b_1=c_2=g_0=g_2=h_2=-1.
\end{array} 
\]
Then $T_{p'}X_4=T_{q'}X_4=\{z_2+z_3=0\}$, and $C_4$ has double points at $p'=[1\!:\!0\!:\!0\!:\!0]$, $q'=[0\!:\!1\!:\!0\!:\!0]$, $x=[0\!:\!0\!:\!1\!:\!-1]$. The quadratic Cremona transformation sends $C_4$ to
\(
C_2=\{w_0^2-2w_0w_1+3w_1^2-2w_0w_2+w_1w_2+w_2^2=0\},
\)
which is smooth and contains $[1\!:\!0\!:\!-1]$, hence rational; so $\mathcal{Q}^1$ is unirational.
\end{Remark}

The automorphisms of $\Tcal_{2,1,1}$ are of the form
\stepcounter{thm}
\begin{equation}\label{eq:aut422}
\begin{array}{rcl}
x_0 &\mapsto& \alpha_{0,0}x_0+\alpha_{0,1}x_1,\\
x_1 &\mapsto& \alpha_{1,0}x_0+\alpha_{1,1}x_1,\\
y_0 &\mapsto& \beta_{0,0}y_0,\\
y_1 &\mapsto& \gamma_{1,1}y_1+\gamma_{1,2}y_2+\delta_{1,1}x_0y_0+\delta_{1,2}x_1y_0,\\
y_2 &\mapsto& \gamma_{2,1}y_1+\gamma_{2,2}y_2+\delta_{2,1}x_0y_0+\delta_{2,2}x_1y_0.
\end{array}
\end{equation}
We set $\Gcal{2}{1}{1}\coloneqq \Aut\!\big(\Tcal_{2,1,1}\big)$.

Assume $\mathcal{Q}^1$ passes through
\(
p=([0\!:\!1],[1\!:\!0\!:\!0])
\)
and
\(
q=([1\!:\!0],[1\!:\!0\!:\!0]).
\)
Then $p'=[1\!:\!0\!:\!0\!:\!0]$, $q'=[0\!:\!1\!:\!0\!:\!0]$, and $a_0=a_4=0$. We have
\[
T_{p'}X_4=\{a_3z_0+b_3z_2+c_3z_3=0\},\qquad
T_{q'}X_4=\{a_1z_1+b_0z_2+c_0z_3=0\}.
\]
Thus the condition $T_{p'}X_4=T_{q'}X_4$ becomes
\stepcounter{thm}
\begin{equation}\label{eq:reqt433222}
a_1=0,\qquad a_3=0,\qquad b_0c_3-c_0b_3=0.
\end{equation}
Set
\(
\mathcal{U}\coloneqq \{a_0=a_4=a_1=a_3=0,\ b_0c_3-c_0b_3=0\}.
\)
In $\PVV{4}{3}{3}{2}{2}{2}$, a general point of the $16$-dimensional quadric $\mathcal{U}\subset\mathbb{P}^{21}$ corresponds to a $\mathcal{Q}^1$ with $T_p\mathcal{Q}^1=T_q\mathcal{Q}^1$, hence unirational by Proposition~\ref{prop:TTUni433222}.

\begin{Proposition}\label{prop:dom433222b}
Consider the action
\[
\Gcal{2}{1}{1}\times \PVV{4}{3}{3}{2}{2}{2} \longrightarrow \PVV{4}{3}{3}{2}{2}{2},\quad
(\varphi,\mathcal{Q}^1)\mapsto \varphi^*\mathcal{Q}^1.
\]
Then the morphism
\[
\chi:\ \Gcal{2}{1}{1}\times \mathcal{U}\longrightarrow \PVV{4}{3}{3}{2}{2}{2},\quad
(\varphi,\mathcal{Q}^1)\mapsto \varphi^*\mathcal{Q}^1
\]
is dominant.
\end{Proposition}

\begin{proof}
Impose \eqref{eq:reqt433222} and write $c_3=\frac{b_3}{b_0}c_0$; view $\Gcal{2}{1}{1}$ as an open in a toric variety as customary. Fix a general $\mathcal{Q}^1\in\mathcal{U}$; restrict to the open $\{\alpha_{0,0}\ne0,\ \beta_{0,0}\ne0\}\cong\mathbb{A}^{11}$; write $f_0,\dots,f_{21}$ for the coefficients of $\varphi^*\mathcal{Q}^1$. The map
\(
f=(f_0/f_{21},\dots,f_{20}/f_{21})
\)
parametrizes the orbit $\Gcal{2}{1}{1}\cdot \mathcal{Q}^1$; the Jacobian of $\chi$ at $\varphi=\mathrm{id}$ has maximal rank, hence $\chi$ is dominant.
\end{proof}

\begin{Remark}\label{rmk:dense433222}
By Proposition~\ref{prop:dom433222b}, $\mathcal{D}_{4,3,3,2,2,2}\coloneqq \chi(\Gcal{2}{1}{1}\times\mathcal{U})$ is dense, i.e.
\(
\overline{\chi(\Gcal{2}{1}{1}\times\mathcal{U})}=\PVV{4}{3}{3}{2}{2}{2}.
\) Thus the subset of $\PVV{4}{3}{3}{2}{2}{2}$ parameterizing unirational conic bundles is unirational.
\end{Remark}

\begin{Proposition}\label{prop:nosec433222}
Let $\mathcal{Q}^1$ be general in $\mathcal{D}_{4,3,3,2,2,2}$, and assume that $k$ is perfect. Then $\mathcal{Q}^1$ has no section and is not rational.
\end{Proposition}

\begin{proof}
It suffices to treat conic bundles general among those with $b_2=b_3=d_0=d_2=0$, the general case then following by Lemma~\ref{sect}. Set $b_2=b_3=d_0=d_2=0$. By Proposition~\ref{diag} the quadratic form
\[
\begin{aligned}
\mathcal{Q} &=
\frac{a_2t^2}{h_0t^2+h_1t+h_2}y_0^2
+ \frac{b_0t^3+b_1t^2+b_2t+b_3}{h_0t^2+h_1t+h_2}y_0y_1
+ \frac{b_0c_0t^3+b_0c_1t^2+b_0c_2t+b_3c_0}{b_0(h_0t^2+h_1t+h_2)}y_0y_2\\
&\qquad + \frac{d_0t^2+d_1t+d_2}{h_0t^2+h_1t+h_2}y_1^2
+ \frac{g_0t^2+g_1t+g_2}{h_0t^2+h_1t+h_2}y_1y_2 + y_2^2
\end{aligned}
\]
is congruent to
\(
\mathcal{Q}'=\frac{a_2t^2}{h_0t^2+h_1t+h_2}z_0^2+\beta_1 z_1^2+\beta_2 z_2^2,
\)
where
\[
\beta_1 = \frac{c \,t^{5} \left(-b_{0}^{2} t^{3}-2 b_{0} b_{1} t^{2}-b_{1}^{2} t +4 c d_{1}\right)}{\left(h_{0} t^{2}+h_{1} t +h_{2}\right)^{3}},
\]
and
\begingroup\small
\[
\begin{aligned}
\beta_{2} &= \frac{1}{(h_{0} t^{2}+h_{1} t +h_{2})^5}\Big(4t^{4} (b_{0}^{2} h_{0} t^{7}+(b_{0}^{2} h_{1}+2 b_{0} b_{1} h_{0}+c_{0}^{2} d_{1}-c_{0} g_{0} h_{0}) t^{6}\\
&\quad+(h_{2} b_{0}^{2}+(2 c_{0} c_{1} d_{1}-c_{1} g_{0} h_{0}+2 b_{1} h_{1}) b_{0} + (b_{1}^{2}-c_{0} g_{1}) h_{0}+c g_{0}^{2}-h_{1} c_{0} g_{0}) t^{5}\\
&\quad+(b_{0}^{2} c_{1}^{2} d_{1}+((-c_{1} g_{1}-c_{2} g_{0}) h_{0}-h_{1} c_{1} g_{0}+2 c_{0} c_{2} d_{1}+2 h_{2} b_{1}) b_{0}+(-4 c d_{1}-c_{0} g_{2}) h_{0}\\
&\quad + (b_{1}^{2}-c_{0} g_{1}) h_{1}+2 c g_{1} g_{0} - h_{2} c_{0} g_{0}) t^{4}+(2 b_{0}^{2} c_{1} c_{2} d_{1}+((-c_{1} g_{2}-c_{2} g_{1}) h_{0}+(-c_{1} g_{1}-c_{2} g_{0}) h_{1}-h_{2} c_{1} g_{0}) b_{0}\\
&\quad + (-4 c d_{1}-c_{0} g_{2}) h_{1}+(b_{1}^{2}-c_{0} g_{1}) h_{2}+c (2 g_{0} g_{2}+g_{1}^{2})) t^{3}+(b_{0}^{2} c_{2}^{2} d_{1}+(-g_{2} c_{2} h_{0}+(-c_{1} g_{2}-c_{2} g_{1}) h_{1}\\
&\quad -h_{2} (c_{1} g_{1}+c_{2} g_{0})) b_{0} + (-4 c d_{1}-c_{0} g_{2}) h_{2}+ 2 c g_{1} g_{2}) t^{2}+((-g_{2} c_{2} h_{1}-h_{2} (c_{1} g_{2}+c_{2} g_{1})) b_{0}+c g_{2}^{2}) t\\
&\quad -b_{0} c_{2} g_{2} h_{2}\Big) \cdot \Big(-\frac{1}{4} b_{0}^{2} t^{3}-\frac{1}{2} b_{0} b_{1} t^{2}-\frac{1}{4} b_{1}^{2} t +c d_{1}\Big).
\end{aligned}
\]
\endgroup
Dividing by $\beta_2$ gives
\(
\mathcal{Q}'=a z_0^2+b z_1^2-z_2^2
\)
with $a=-\frac{a_2t^2}{\beta_2(h_0t^2+h_1t+h_2)}$, $b=-\frac{\beta_1}{\beta_2}$. Here $\beta_2$ has a zero of order $4$ at $t=0$, $\beta_1$ a zero of order $5$, so $a$ has a pole of order $2$ and $b$ a simple zero at $t=0$. Therefore
\(
\partial_t^2(a,b)=\overline{ab^2}.
\)
Modulo $t$, the class of $ab^2$ equals
\(
-\frac{c^2h_2^5}{4b_0^3c_2^3g_2^3d_1},
\)
which is not a square in $k$; hence $\partial_t^2(a,b)\neq 0$ and $\mathcal{Q}^1$ has no section.
\end{proof}

\begin{Proposition}\label{prop:min433222}
Assume that $k$ is not algebraically closed and perfect. A general conic bundle $\mathcal{Q}^1\in \mathcal{D}_{4,3,3,2,2,2}$ is minimal.
\end{Proposition}

\begin{proof}
It suffices to show that all fibers of $\mathcal Q^1$ are irreducible. For a general $\mathcal{Q}^1\in\mathcal{D}_{4,3,3,2,2,2}$ the discriminant has the form
\[
\begin{aligned}
\delta_{\mathcal{Q}^1} &=  c \,t^{2} (d_{0} t^{2}+d_{1} t +d_{2}) (h_{0} t^{2}+h_{1} t +h_{2})
-\frac{c \,t^{2} (g_{0} t^{2}+g_{1} t +g_{2})^{2}}{4}\\
&\quad -\frac{(b_{0} t^{3}+b_{1} t^{2}+b_{2} t +b_{3})^{2} (h_{0} t^{2}+h_{1} t +h_{2})}{4}
+\frac{(c_{1} b_{0} t^{2}+c_{0} t^{3}+c_{2} b_{0} t +c_{0} b_{3}) (g_{0} t^{2}+g_{1} t +g_{2})}{4}\\
&\quad -\frac{(c_{1} b_{0} t^{2}+c_{0} t^{3}+c_{2} b_{0} t +c_{0} b_{3})^{2} (d_{0} t^{2}+d_{1} t +d_{2})}{4},
\end{aligned}
\]
which is irreducible over $k$ for general parameters. Even if a geometric fiber were reducible, for general parameters its components would not be defined over $k$ (since $k$ is not algebraically closed), hence the fiber is irreducible over $k$. Thus $\rho(\mathcal{Q}^1/W)=1$ (extremal), and by Proposition~\ref{prop:nosec433222} there are no sections, hence minimal.
\end{proof}

\begin{thm}\label{thm:main433222}
Assume $k$ is $C_1$ and perfect. There exists a dense subset
\(
\mathcal{D}_{4,3,3,2,2,2}\subset \PVV{4}{3}{3}{2}{2}{2}
\)
whose general point corresponds to a minimal unirational conic bundle.
\end{thm}
\begin{proof}
Combine Propositions~\ref{prop:min433222}, \ref{prop:nosec433222} with the unirationality established in Remark~\ref{rmk:dense433222}.
\end{proof}

\section{Conic bundles of type \texorpdfstring{$\TypeTwo{2}{2}{0}$}{(2,2,0)}}\label{sec:442420}
Consider now conic bundles on the form
\[
\mathcal{Q}^1
=\{\sigma_{0,0}y_0^2+\sigma_{0,1}y_0y_1+\sigma_{0,2}y_0y_2+\sigma_{1,1}y_1^2+\sigma_{1,2}y_1y_2+\sigma_{2,2}y_2^2=0\}\subset \Tcal_{2,2,0},
\]
with $\sigma_{0,0},\sigma_{0,1},\sigma_{1,1}\in k[x_0,x_1]_4$, $\sigma_{0,2},\sigma_{1,2}\in k[x_0,x_1]_2$, and $\sigma_{2,2}\in k[x_0,x_1]_0$ (a constant). Then $\phi(\mathcal{Q}^1)$ for
\[
\phi\colon \Tcal_{2,2,0}\to \mathbb{P}^3,\quad (x_0,x_1,y_0,y_1,y_2)\mapsto [\,x_0:y_1:y_2:1\,],
\]
is a sextic surface $X_6\subset\mathbb{P}^3$ with multiplicity $4$ along $\Lambda_4\coloneqq \{z_0=z_3=0\}$ and multiplicity $2$ along $\Lambda_2\coloneqq \{z_1=z_3=0\}$. At $q=[0,0,1,0]=\Lambda_4\cap\Lambda_2$, the tangent cone is $TC_q(X_6)=\{z_3^4=0\}$. Write
\stepcounter{thm}
\begin{equation}\label{eq:pol442420}
\begin{aligned}
\sigma_{0,0} &= a_0x_0^4 + a_1x_0^3x_1 + a_2x_0^2x_1^2 + a_3x_0x_1^3 + a_4x_1^4,\\
\sigma_{0,1} &= b_0x_0^4 + b_1x_0^3x_1 + b_2x_0^2x_1^2 + b_3x_0x_1^3 + b_4x_1^4,\\
\sigma_{0,2} &= c_0x_0^2 + c_1x_0x_1 + c_2x_1^2,\\
\sigma_{1,1} &= d_0x_0^4 + d_1x_0^3x_1 + d_2x_0^2x_1^2 + d_3x_0x_1^3 + d_4x_1^4,\\
\sigma_{1,2} &= g_0x_0^2 + g_1x_0x_1 + g_2x_1^2,\\
\sigma_{2,2} &= h_0.
\end{aligned}
\end{equation}

\begin{Proposition}\label{prop:TTUni442420}
Assume $k$ is $C_1$. If there exists a point $p\in \mathcal{Q}^1$ mapping to $p'=\phi(p)\in X_6$ such that $p'\notin\Lambda_2$, the tangent plane $T_{p'}X_6$ contains $\Lambda_2$, and $X_6$ is otherwise general, then $\mathcal{Q}^1$ is unirational.
\end{Proposition}

\begin{proof}
We have $C_6\coloneqq X_6\cap T_{p'}X_6=2\Lambda_2 + C_4$, with $C_4$ a quartic having an ordinary double point at $p'$ and a cusp at $q$ with $TC_q(C_4)=\{z_3^2=0\}$. Its arithmetic genus is $0$, hence $C_4$ is rational over $\bar{k}$. To obtain a $k$-rational model, take the Cremona transformation $\psi$ on $T_{p'}X_6$ given by conics through $p'$ and $q$ with fixed tangent line $L=(z_3=0)\cap T_{p'}X_6$. Then $\overline{\psi(C_4)}=C_2\subset\mathbb{P}^2_k$ is a smooth conic; since $k$ is $C_1$, $C_2(k)\neq\varnothing$, hence $C_2$ (and $C_4$) are $k$-rational. Proposition~\ref{Enr} gives the claim.
\end{proof}

The automorphisms of $\Tcal_{2,2,0}$ are of the form
\stepcounter{thm}
\begin{equation}\label{eq:aut220}
\begin{array}{rcl}
x_0&\mapsto&\alpha_{0,0}x_0+\alpha_{0,1}x_1,\quad
x_1\mapsto \alpha_{1,0}x_0+\alpha_{1,1}x_1,\\
y_0&\mapsto&\beta_{0,0}y_0+\beta_{0,1}y_1,\quad
y_1\mapsto \beta_{1,0}y_0+\beta_{1,1}y_1,\\
y_2&\mapsto&\gamma_{2,2}y_2+\sum_{i=1}^3\big(\delta_{1,i}\,x_0^{3-i}x_1^{i-1}y_0+\theta_{1,i}\,x_0^{3-i}x_1^{i-1}y_1\big).
\end{array}
\end{equation}
We set $\Gcal{2}{2}{0}\coloneqq \Aut\!\big(\Tcal_{2,2,0}\big)$.

If $\mathcal{Q}^1$ passes through $p=([1\!:\!0],[1\!:\!0\!:\!0])$, then $p'=[0\!:\!0\!:\!0\!:\!1]$ and $a_4=0$. The condition $\Lambda_2\subset T_{p'}X_6$ reads
\stepcounter{thm}
\begin{equation}\label{eq:reqt442420}
a_3=0,\qquad c_2=0.
\end{equation}
Set $\mathcal{U}\coloneqq \{a_3=a_4=c_2=0\}$. In $\PVV{4}{4}{2}{4}{2}{0}$, a general point of the $18$-dimensional linear subspace $\mathcal{U}\subset\mathbb{P}^{21}$ yields $\Lambda_2\subset T_{p'}\phi(\mathcal{Q}^1)$, hence unirational by Proposition~\ref{prop:TTUni442420}.

\begin{Proposition}\label{prop:dom442420}
With notation as above, the map
\[
\chi:\ \Gcal{2}{2}{0}\times \mathcal{U}\longrightarrow \PVV{4}{4}{2}{4}{2}{0},\quad
(\varphi,\mathcal{Q}^1)\mapsto \varphi^*\mathcal{Q}^1,
\]
is dominant.
\end{Proposition}

\begin{proof}
Impose \eqref{eq:reqt442420}. Fix $\mathcal{Q}^1\in\mathcal{U}$ general, let $f_0,\dots,f_{21}$ be the coefficients of $\varphi^*\mathcal{Q}^1$, and restrict to the open $\{\alpha_{0,0}\ne0,\ \beta_{0,0}\ne0\}\cong\mathbb{A}^{13}$ in $\Gcal{2}{2}{0}$. The map $f=(f_0/f_{21},\dots,f_{20}/f_{21})$ parametrizes $\Gcal{2}{2}{0}\cdot\mathcal{Q}^1$; the Jacobian of $\chi$ at $\mathrm{id}$ has maximal rank, hence $\chi$ is dominant.
\end{proof}

\begin{Remark}\label{rmk:dense442420}
We have $\overline{\chi(\Gcal{2}{2}{0}\times\mathcal{U})}=\PVV{4}{4}{2}{4}{2}{0}$, and $\mathcal{D}_{(4,4,2,4,2,0)} = \chi(\Gcal{2}{2}{0}\times\mathcal{U})$ is Zariski dense.
\end{Remark}

\begin{Proposition}\label{prop:elld442420}
Let $\pi\colon\mathcal{Q}^1\to\mathbb{P}^1$ be a smooth conic bundle of type $\TypeTwo{2}{2}{0}$ over a number field $k$. Then $\mathcal{Q}^1(k)$ is dense in $\mathcal{Q}^1$ if and only if $\mathcal{Q}^1$ has a $k$-point on a smooth fiber of $\pi$.
\end{Proposition}

\begin{proof}
Consider
\[
\phi\colon \mathcal{Q}^1\to\mathbb{P}^1,\quad (x_0,x_1,y_0,y_1,y_2)\mapsto [y_0:y_1].
\]
A general fiber $\Gamma_y$ is a double cover of $\mathbb{P}^1$ branched in four points, hence a genus-one curve. If some $\Gamma_{\bar y}$ splits over $k$ as two components $\Gamma_1\cup\Gamma_2$, then each $\Gamma_i$ is a section of $\pi$, so $\mathcal{Q}^1$ is rational. Otherwise all fibers of $\phi$ are irreducible. Let $p\in\mathcal{Q}^1(k)$ lie on a smooth fiber $C_p$ of $\pi$. Base-change by $C_p\to \mathbb{P}^1$:
\[
\begin{tikzcd}
S\coloneqq \mathcal{Q}^1\times_{\mathbb{P}^1}C_p \arrow[r] \arrow[d, "\psi"'] & \mathcal{Q}^1 \arrow[d, "\phi"] \\
C_p \arrow[r] & \mathbb{P}^1
\end{tikzcd}
\]
The $2$-section $C_p$ of $\phi$ induces a section of $\psi\colon S\to C_p\cong\mathbb{P}^1$, giving an elliptic fibration with irreducible fibers and a rational curve on $S$. By \cite[Corollary~8.2(a)]{HT00}, $S(k)$ is dense; since $S$ dominates $\mathcal{Q}^1$, the claim follows.
\end{proof}

\begin{Proposition}\label{prop:elld442420ns}
Let $k$ be a perfect, non algebraically closed field. A general conic bundle in $\mathcal{D}_{(4,4,2,4,2,0)}$ has no section and is not rational. 
\end{Proposition}
\begin{proof}
Putting the equation of the conic bundle in diagonal form we see that there is a residue equal to $4d_0h_0 - g_0^2$ which, in general, is not a square.
\end{proof}

\section{Conic bundles of type \texorpdfstring{$\TypeTwo{3}{1}{0}$}{(3,1,0)}}\label{sec:643210}
We now consider conic bundles of type $(3,1,0)$, that is
\[
\mathcal{Q}^1=\{\sigma_{0,0}y_0^2+\sigma_{0,1}y_0y_1+\sigma_{0,2}y_0y_2+\sigma_{1,1}y_1^2+\sigma_{1,2}y_1y_2+\sigma_{2,2}y_2^2=0\}\subset \Tcal_{3,1,0},
\]
with $\sigma_{0,0}\in k[x_0,x_1]_6$, $\sigma_{0,1}\in k[x_0,x_1]_4$, $\sigma_{0,2}\in k[x_0,x_1]_3$, $\sigma_{1,1}\in k[x_0,x_1]_2$, $\sigma_{1,2}\in k[x_0,x_1]_1$, $\sigma_{2,2}\in k$.

\begin{Proposition}\label{prop:case643210}
A general conic bundle $\mathcal{Q}^1$ of type $(3,1,0)$ over a $C_1$ field is unirational.
\end{Proposition}

\begin{proof}
Intersect with $H=\{y_0=0\}$ to obtain:
\[
\mathcal{Q}^1\cap H=\{\sigma_{1,1}y_1^2+\sigma_{1,2}y_1y_2+\sigma_{2,2}y_2^2=0\}\to \mathbb{P}^1_{x_0,x_1}.
\]
For general $\sigma_{1,1},\sigma_{1,2},\sigma_{2,2}$ this is birational to a smooth conic over a $C_1$ field, hence it has a point and it is rational. Thus $\mathcal{Q}^1\cap H$ is a rational multisection, and Proposition~\ref{Enr} applies.
\end{proof}

\begin{Proposition}\label{nr643}
A general conic bundle $\mathcal{Q}^1$ of type $(3,1,0)$, over a perfect non-algebraically closed field, is minimal and non rational.
\end{Proposition}
\begin{proof}
Putting the equation of $\mathcal{Q}^1$ in normal form one sees that there is a non trivial residue.
\end{proof}

\section{Conic bundles of type \texorpdfstring{$\TypeTwo{4}{0}{0}$}{(4,0,0)}}\label{sec:844000}
Consider
\[
\mathcal{Q}^1
=\{\sigma_{0,0}y_0^2+\sigma_{0,1}y_0y_1+\sigma_{0,2}y_0y_2+\sigma_{1,1}y_1^2+\sigma_{1,2}y_1y_2+\sigma_{2,2}y_2^2=0\}\subset \Tcal_{4,0,0},
\]
with $\sigma_{0,0}\in k[x_0,x_1]_8$, $\sigma_{0,1},\sigma_{0,2}\in k[x_0,x_1]_4$, and $\sigma_{1,1},\sigma_{1,2},\sigma_{2,2}\in k$. Write
\stepcounter{thm}
\begin{equation}\label{eq:pol844000}
\begin{aligned}
\sigma_{0,0} &= a_0x_0^8 + a_1x_0^7x_1 + \cdots + a_8x_1^8,\\
\sigma_{0,1} &= b_0x_0^4 + b_1x_0^3x_1 + \cdots + b_4x_1^4,\\
\sigma_{0,2} &= d_0x_0^4 + d_1x_0^3x_1 + \cdots + d_4x_1^4,\\
\sigma_{1,1} &= c_0,\qquad
\sigma_{1,2} = c_1,\qquad
\sigma_{2,2} = c_2 .
\end{aligned}
\end{equation}

\begin{Proposition}\label{prop:Me844000}
Assume $c_2\in k^\times$, $c_1^2-4c_0c_2\in k^\times$, and that the leading coefficient of $-\frac{4}{c_2}\delta_{\mathcal{Q}^1}$ is a square in $k$. Assume furthermore that $k$ is perfect. Then $\mathcal{Q}^1$ is unirational.
\end{Proposition}

\begin{proof}
On the chart $x_1\ne0$ set $t=x_0/x_1$, and write as before $\delta_{\mathcal Q^1} = \delta_{\mathcal Q^1}(t)$ for the discriminant polynomial of $\mathcal Q^1$. Apply Proposition~\ref{diag} with $\alpha_0=\sigma_{2,2}$, $\alpha_1=\sigma_{1,2}$, $\alpha_3=\sigma_{1,1}$ to obtain a diagonal model
\(
\mathcal{Q}': P z_2^2+\frac{1}{c_1^2-4c_0c_2}z_0^2-z_1^2=0,
\)
where $P=-\frac{4}{\alpha_0}\delta_{\mathcal{Q}^1}$. Let
\[
\begin{aligned}
A &= \frac{-4a_1c_0c_2 + a_1c_1^2 + 2b_0b_1c_2 - b_0c_1d_1 - b_1c_1d_0 + 2c_0d_0d_1}{c_2},\\
B &= \frac{(-4a_0c_0 + b_0^2)c_2 + a_0c_1^2 - b_0c_1d_0 + d_0^2c_0}{c_2},
\end{aligned}
\]
and set $R(u)\coloneqq P\!\left(-\frac{A}{8B}-u\right)$; here $B\ne0$ is the leading coefficient of $P$, and $R$ is a polynomial of degree $8$ in $u$ with no $u^7$-term. If
\[
\frac{c_2}{(-4a_0c_0+b_0^2)c_2+a_0c_1^2-b_0c_1d_0+d_0^2c_0}
=\xi^2\in k^{*2},
\] (here $k^{*2}$ denotes the set of all nonzero squares in $k$),
we obtain
\(
Tw_2^2 + c\,w_0^2 - w_1^2=0
\)
with $T$ monic of degree $8$ and $c=\frac{1}{(c_1^2-4c_0c_2)B}\in k^\times$. Unirationality follows from \cite{Mes94}.
\end{proof}

\begin{Lemma}\label{lem:d8}
In $\PVV{8}{4}{4}{0}{0}{0}$, the subset
\[
\mathcal{U}^{\delta}\coloneqq \Big\{\mathcal{Q}^1:\ -\tfrac{4}{c_2}\delta_{\mathcal{Q}^1}\ \text{is a square in }k\Big\}
\]
is dense.
\end{Lemma}

\begin{proof}
Consider the incidence scheme
\(
\mathcal{I}=\{(\xi,\mathcal{Q}^1):\ \xi^2+\tfrac{4}{c_2}\delta_{\mathcal{Q}^1}=0\}
\subset k\times \PVV{8}{4}{4}{0}{0}{0}
\)
with projection $\pi_2$; then $\mathcal{U}^{\delta}=\pi_2(\mathcal{I})$. Writing
\[
-\tfrac{4}{c_2}\delta_{\mathcal{Q}^1}=\frac{(-4a_0c_0 + b_0^2)c_2 + a_0c_1^2 - b_0c_1d_0 + d_0^2c_0}{c_2},
\]
let
\[
Z=\{\xi^2c_2 -(b_0^2-4a_0c_0)c_2 - a_0c_1^2 + b_0c_1d_0 - d_0^2c_0 = 0\}\subset \mathbb{A}^{23}_k,
\]
be the cone over the cubic
\[
W\subset \mathbb{P}^6_{(a_0,b_0,c_0,c_1,c_2,d_0,\xi)}.
\]
The point $[1\!:\!0\!:\!0\!:\!0\!:\!0\!:\!0\!:\!0]\in W$ is a double point, so $W$ is rational and $W(k)$ dense; hence $Z(k)$, and therefore $\mathcal{I}(k)$, are dense. Thus $\mathcal{U}^{\delta}$ is dense.
\end{proof}

The automorphisms of $\Tcal_{4,0,0}$ are of the form
\stepcounter{thm}
\begin{equation}\label{eq:aut800}
\begin{array}{rcl}
x_0&\mapsto&\alpha_{0,0}x_0+\alpha_{0,1}x_1,\quad
x_1\mapsto \alpha_{1,0}x_0+\alpha_{1,1}x_1,\\
y_0&\mapsto&\beta_{0,0}y_0,\\
y_1&\mapsto&\gamma_{1,1}y_1+\gamma_{1,2}y_2+\sum_{i=1}^5 \delta_{1,i}\,x_0^{5-i}x_1^{i-1}y_0,\\
y_2&\mapsto&\gamma_{2,1}y_1+\gamma_{2,2}y_2+\sum_{i=1}^5 \delta_{2,i}\,x_0^{5-i}x_1^{i-1}y_0.
\end{array}
\end{equation}
Set $\Gcal{4}{0}{0}\coloneqq \Aut\!\big(\Tcal_{4,0,0}\big)$.

\begin{Proposition}\label{prop:c2zero}
Assume $k$ is $C_1$. Let $\mathcal{U}\subset \PVV{8}{4}{4}{0}{0}{0}$ be the hyperplane $c_2=0$. Then a general $\mathcal{Q}^1\in\mathcal{U}$ is rational, and
\[
\chi:\ \Gcal{4}{0}{0}\times \mathcal{U}\longrightarrow \PVV{8}{4}{4}{0}{0}{0},\quad
(\varphi,\mathcal{Q}^1)\mapsto \varphi^*\mathcal{Q}^1,
\]
has Zariski dense image.
\end{Proposition}

\begin{proof}
The image of
\(
\phi\colon\Tcal_{4,0,0}\to\mathbb{P}^3,\ (x_0,x_1,y_0,y_1,y_2)\mapsto [x_0:y_1:y_2:1],
\)
is a surface $X_8\subset\mathbb{P}^3$ of degree $8$ with a $6$-fold line $\Lambda_6=\{z_0=z_3=0\}$. At $q=[0,0,1,0]\in \Lambda_6$, $TC_q(X_8)=\{z_3^6=0\}$. 

A general plane $H_q$ through $q$ meets $X_8$ in an irreducible curve $C_8$ with multiplicity $6$ at $q$ and $TC_q(C_8)=\{z_3^6=0\}$. Imposing $c_2=0$ forces a point of multiplicity $7$ at $q$ on a general $C_8$, which is $k$-rational by projection from $q$. 

The conic fibration is induced by the pencil of planes through $\Lambda_6$, hence $C_8$ is a rational section, and Proposition~\ref{Enr} gives rationality. 

For density: fix general $\mathcal{Q}^1\in\mathcal{U}$, restrict to $\{\alpha_{0,0}\ne0,\beta_{0,0}\ne0,\gamma_{1,1}\ne0\}\cong\mathbb{A}^{17}$. The Jacobian of the coefficient map at the identity has rank $21=\dim \PVV{8}{4}{4}{0}{0}{0}$.
\end{proof}

\begin{Proposition}\label{prop:U12}
Assume $k$ is $C_1$. Let $\mathcal{U}^{12}\subset \PVV{8}{4}{4}{0}{0}{0}$ be the codimension $9$ linear subspace defined by
\[
\left\{
\begin{array}{l}
a_{0} = a_{1} - a_{2} + a_{3} - a_{4} + a_{5} - a_{6} + a_{7}, \\[1pt]
a_{1} = a_{5} + \tfrac{1}{2}b_{2} + \tfrac{1}{2}c_{2} + \tfrac{1}{2}d_{0} + \tfrac{1}{2}g_{0} + \tfrac{1}{2}h_{0}, \\[1pt]
a_{2} = -2a_{4} - 3a_{6} - \tfrac{1}{2}b_{2} - \tfrac{1}{2}c_{2} + \tfrac{1}{4}d_{0} + \tfrac{1}{4}g_{0} + \tfrac{1}{4}h_{0}, \\[1pt]
a_{3} = -2a_{5} - \tfrac{1}{2}b_{2} - \tfrac{1}{2}c_{2}, \\[4pt]
a_{7} = 0,\quad a_{8} = 0, \\[1pt]
b_{0} = -g_{0} - 2a_{1} - b_{1} - b_{2} - c_{1} - c_{0} - c_{4} - b_{4} - 2a_{7} - h_{0} - d_{0} - 2a_{5} - c_{3} - b_{3} - 2a_{3} - c_{2}, \\[1pt]
b_{1} = -b_{3} - c_{1} - c_{3} - d_{0} - g_{0} - h_{0}, \\[1pt]
b_{4} = -c_{4}.
\end{array}
\right.
\]
Then a general $\mathcal{Q}^1\in\mathcal{U}^{12}$ is unirational, and
\[
\chi\colon \Gcal{4}{0}{0}\times \mathcal{U}^{12}\to \PVV{8}{4}{4}{0}{0}{0}
\]
has Zariski dense image.
\end{Proposition}
\begin{proof}
Consider the rational map
\[
\iota\colon \Tcal_{4,0,0}\dashrightarrow \PP^3_{[z_0:z_1:z_2:w]},\qquad
(x_0,x_1,y_0,y_1,y_2)\longmapsto [\,z_0:z_1:z_2:w\,]=[\,x_0:y_1:y_2:1\,].
\]
The image $X_8\coloneqq \overline{\iota(\mathcal{Q}^1)}\subset\PP^3$ is a degree-$8$ surface having the $6$–fold line
\[
\Lambda_6=\{z_0=w=0\}.
\]
Consider the plane
\[
H\coloneqq \{z_1-z_2=0\}\subset\PP^3.
\]
A direct substitution using the relations defining $\Ucal^{12}$ shows that $H$ is tangent to $X_8$ at the three points
\[
p_1=[1:0:0:-1],\qquad p_2=[0:0:0:1],\qquad p_3=[1:1:1:1],
\]
and that the intersection
\[
C\coloneqq X_8\cap H
\]
is a plane octic with ordinary double points at $p_1,p_2,p_3$ and a point of multiplicity $6$ at
\[
q=H\cap\Lambda_6=[0:1:1:0].
\]
Moreover, the tangent cone is given by $TC_q(C)=\{w^6=0\}$ inside the plane $H=\{z_1=z_2\}$, and $C$ is transverse to the conic fibration induced by the pencil of planes through $\Lambda_6$.

Working in the plane $H$ we keep the ambient coordinates $[z_0:z_1:z_2:w]$ and impose $z_1=z_2$. Then $C$ is cut by $z_1-z_2=0$ and by the homogeneous octic
\[
\begin{aligned}
&z_0^8 + \frac{a_5 + \tfrac{1}{2}b_2 + \tfrac{1}{2}c_2 + \tfrac{1}{2}d_0 + \tfrac{1}{2}g_0 + \tfrac{1}{2}h_0}{\Delta}\,z_0^7w
+ \frac{-2a_4 - 3a_6 - \tfrac{1}{2}b_2 - \tfrac{1}{2}c_2 + \tfrac{1}{4}d_0 + \tfrac{1}{4}g_0 + \tfrac{1}{4}h_0}{\Delta}\,z_0^6w^2 \\
&\quad + \frac{-2a_5 - \tfrac{1}{2}b_2 - \tfrac{1}{2}c_2}{\Delta}\,z_0^5w^3
+ \frac{-b_2 - c_0 - c_2 - d_0 - g_0 - h_0}{\Delta}\,z_0^4 z_1 w^3
+ \frac{c_0}{\Delta}\,z_0^4 z_2 w^3
+ \frac{a_4}{\Delta}\,z_0^4 w^4 \\
&\quad + \frac{-c_1 - c_3 - d_0 - g_0 - h_0}{\Delta}\,z_0^3 z_1 w^4
+ \frac{c_1 + c_3}{\Delta}\,z_0^3 z_2 w^4
+ \frac{a_5}{\Delta}\,z_0^3 w^5
+ \frac{b_2}{\Delta}\,z_0^2 z_1 w^5
+ \frac{c_2}{\Delta}\,z_0^2 z_2 w^5 \\
&\quad + \frac{a_6}{\Delta}\,z_0^2 w^6
+ \frac{d_0}{\Delta}\,z_1^2 w^6
+ \frac{g_0}{\Delta}\,z_1 z_2 w^6
+ \frac{h_0}{\Delta}\,z_2^2 w^6
- \frac{c_4}{\Delta}\,z_1 w^7
+ \frac{c_4}{\Delta}\,z_2 w^7,
\end{aligned}
\]
where for brevity we set
\[
\Delta\coloneqq a_4 + 2a_6 + \tfrac{1}{2}b_2 + \tfrac{1}{2}c_2 + \tfrac{1}{4}d_0 + \tfrac{1}{4}g_0 + \tfrac{1}{4}h_0.
\]
Identify $H\cong\PP^2$ via the linear map
\[
j:\ \PP^2_{[w_0:w_1:w_2]}\longrightarrow H\subset\PP^3,\qquad
[w_0:w_1:w_2]\longmapsto [\,z_0:z_1:z_2:w\,]=[\,w_0:w_1:w_1:w_2\,].
\]
Under $j$, the points $c=[0:1:0]$, $n_1=[0:0:1]$, and $n_3=[1:0:-1]$ of $\PP^2$ map to $p_3,p_2,p_1$ respectively. Consider the quadratic Cremona transformation on $\PP^2$ with base points
\[
B_1=\{c,n_1,n_3\},\qquad \phi_1:\PP^2\dashrightarrow\PP^2,
\]
given by the complete linear system of conics through $B_1$. Applying $\phi_1$ to $C$ (viewed in $\PP^2$ via $j^{-1}$) we obtain a degree $6$ plane curve $C_1=\phi_1(C)$ with explicit equation
\[
\begin{aligned}
& w_0^2 w_1^4
+ \frac{-2a_4 + a_5 - 4a_6 - \tfrac{1}{2}b_2 - \tfrac{1}{2}c_2}{\Delta}\,w_0^2 w_1^3 w_2
+ \frac{a_4 - 2a_5 + 3a_6}{\Delta}\,w_0^2 w_1^2 w_2^2
+ \frac{a_5 - 2a_6}{\Delta}\,w_0^2 w_1 w_2^3 \\
&\quad + \frac{-b_2 - c_2 - d_0 - g_0 - h_0}{\Delta}\,w_0 w_1^2 w_2^3
+ \frac{a_6}{\Delta}\,w_0^2 w_2^4
+ \frac{b_2 + c_2}{\Delta}\,w_0 w_1 w_2^4
+ \frac{d_0 + g_0 + h_0}{\Delta}\,w_2^6.
\end{aligned}
\]
Next consider the quadratic Cremona transformation
\[
\phi_2:\PP^2\dashrightarrow\PP^2
\]
with base points
\[
B_2=\{q_1,q_2,q_3\}=\{[2:1:1],\ [0:1:0],\ [1:0:0]\},
\]
again given by the complete linear system of conics through $B_2$. Applying $\phi_2$ to $C_1$ yields a quartic $C_2=\phi_2(C_1)$ with equation
\[
\begin{aligned}
& w_0^4
+ \frac{-2a_4 + a_5 - 4a_6 - \tfrac{1}{2}b_2 - \tfrac{1}{2}c_2}{\Delta}\,w_0^3 w_1
+ \frac{a_4 - 2a_5 + 3a_6}{\Delta}\,w_0^2 w_1^2
+ \frac{a_5 - 2a_6}{\Delta}\,w_0 w_1^3
+ \frac{a_6}{\Delta}\,w_1^4 \\
&\quad - 4 w_0^3 w_2
+ \frac{8a_4 - 2a_5 + 16a_6 + 3b_2 + 3c_2 + d_0 + g_0 + h_0}{\Delta}\,w_0^2 w_1 w_2
\\
&\quad + \frac{-4a_4 + 4a_5 - 8a_6 - b_2 - c_2}{\Delta}\,w_0 w_1^2 w_2  - \frac{2a_5}{\Delta}\,w_1^3 w_2
+ 4 w_0^2 w_2^2 - 8 w_0 w_1 w_2^2 + 4 w_1^2 w_2^2.
\end{aligned}
\]
Finally, apply the explicit quadratic map
\[
\phi_3:\PP^2\dashrightarrow\PP^2,\qquad
[w_0:w_1:w_2]\longmapsto\big[-2(w_0-w_1)w_2+w_0^2,\ \ w_1^2,\ \ w_0w_1\big].
\]
This is a birational quadratic map with (infinitely near) base supported on the line $\{w_2=0\}$ at the coordinate points; it preserves the rationality class and simplifies the quartic equation. The image $C_3=\phi_3(C_2)$ is the conic with equation
\[
\begin{aligned}
& w_0^2
- \frac{a_5}{\Delta}\,w_0 w_1
+ \frac{a_6}{\Delta}\,w_1^2
+ \frac{-2a_4 + a_5 - 4a_6 - \tfrac{1}{2}b_2 - \tfrac{1}{2}c_2}{\Delta}\,w_0 w_2
+ \frac{a_5 - 2a_6}{\Delta}\,w_1 w_2
+ \frac{a_4 - a_5 + 3a_6}{\Delta}\,w_2^2.
\end{aligned}
\]
Since $k$ is $C_1$, every smooth conic over $k$ has a $k$-point; hence $C_3$ is $k$-rational. Being obtained from $C$ by a sequence of quadratic Cremona transformations, $C$ is $k$-rational as well.

The curve $C\subset H$ is a rational $2$–section of the conic fibration $X_8\dashrightarrow\PP^1$ induced by the planes through $\Lambda_6$. Therefore, by Proposition~\ref{Enr}, the original conic bundle $\mathcal{Q}^1$ is unirational.

For the dominance statement, view $\Gcal{4}{0}{0}=\Aut(\Tcal_{4,0,0})$ as a Zariski open of an affine space via \eqref{eq:aut800}. Fix a general $\mathcal{Q}^1\in\Ucal^{12}$ and consider the coefficient map
\[
\chi\colon \Gcal{4}{0}{0}\times \Ucal^{12}\to \PVV{8}{4}{4}{0}{0}{0},\qquad (\varphi,\mathcal{Q}^1)\mapsto \varphi^*\mathcal{Q}^1.
\]
Restricting to the open where the linear parts are invertible, the differential of $\chi$ at the identity has full rank $21=\dim\PVV{8}{4}{4}{0}{0}{0}$. Hence $\chi$ is dominant and its image is Zariski dense.
\end{proof}

\begin{Proposition}\label{prop:min844}
Assume that $k$ is not algebraically closed. Then a general conic bundle $\mathcal{Q}^1\in\mathcal{U}^{\delta}$ is minimal.
\end{Proposition}
\begin{proof}
Consider
\(
\mathcal{Q}^1=\{P(t)y_0^2+cy_1^2-y_2^2=0\}\subset \mathbb{A}^1\times\mathbb{P}^2,
\)
with $P(t)=tQ(t)$, $Q(0)\ne0$, and $c\notin k^{*2}$. The fiber over $t=0$ is geometrically reducible ($cy_1^2-y_2^2 = 0$) but its components are not defined over $k$; for general $Q$, the remaining seven singular fibers behave similarly. Hence $\rho(\mathcal{Q}^1/W)=1$. With $a=P(t)$, $b=c$ in Section \ref{Bra}, we have $v_t(a)=1$, so
\(
\partial_t^2(a,b)=\overline{1/c}\ne 0
\)
(as $c$ is not a square), showing there are no sections. Thus $\mathcal{Q}^1$ is minimal.
\end{proof}

Even though Proposition~\ref{prop:c2zero} gives a Zariski dense set of rational bundles in $\PVV{8}{4}{4}{0}{0}{0}$, there is also a Zariski dense set parameterizing non rational bundles.

\begin{Lemma}\label{lem:nonrationalCB844}
Assume that $k$ is perfect. A general conic bundle in $\mathcal{Q}^1\in\mathcal{U}^{12}$ is non rational.
\end{Lemma}
\begin{proof}
Compute the Brauer residues as in Section \ref{Bra}. For the conic $ax^2+by^2=z^2$ with $a,b\in k(\mathbb{P}^1)^*$, residues may be nonzero only at zeros/poles of $a$ and $b$. For a general element of $\PVV{8}{4}{4}{0}{0}{0}$ we have
$$
\begin{array}{ll}
a(t) = & \frac{4c_2}{(4c_0c_2 - c_1^2)P(t)}, \\ 
b(t) = & \frac{4c_2}{P(t)},
\end{array} 
$$
where $P(t)$ is a polynomial. Since for $c_0,c_1,c_2$ general $4c_0c_2 - c_1^2$ is not a square in $k$ we conclude that a general conic bundle in $\PVV{8}{4}{4}{0}{0}{0}$ and also in $\mathcal{U}^{12}$ is non rational. 
\end{proof}

\begin{Proposition}\label{prop:844NotRational}
If $k$ is perfect, there is a Zariski dense subset of $\PVV{8}{4}{4}{0}{0}{0}$ parametrizing non rational conic bundles.
\end{Proposition}

\begin{proof}
As before, the Jacobian rank argument shows that the image of
\(
\Gcal{4}{0}{0}\times \mathcal{U}^{12} \subset \PVV{8}{4}{4}{0}{0}{0}
\) under $\chi$
is dense; combine with Lemma~\ref{lem:nonrationalCB844}.
\end{proof}

\begin{thm}\label{thm:main844}
Assume $k$ is a perfect, not algebraically closed $C_1$ field. Then there exist a dense subset $\mathcal{D}_{8,4,4,0,0,0}:=\chi(\Gcal{4}{0}{0} \times \mathcal{U}^{12}) \subset \PVV{8}{4}{4}{0}{0}{0}$ parametrizing minimal unirational but not rational conic bundles.
\end{thm}
\begin{proof}
Combine Propositions~\ref{prop:Me844000}, \ref{prop:844NotRational}, Lemma~\ref{lem:d8}, and Proposition~\ref{prop:min844}.
\end{proof}

\section{Deformations of the conic bundles and of the ambient projective bundles}\label{defs}

On \(\mathbb{P}^1\) every vector bundle splits as \(\mathcal{E}\simeq\bigoplus_{i=1}^3 \mathcal O(a_i)\) with \(a_1\ge a_2\ge a_3\).
The projectivization \(\mathbb{P}(\mathcal{E})\) depends only on the isomorphism class of \(\mathcal{E}\) up to twist, and its deformation class only on \(c_1(\mathcal{E})\bmod 3\). 
Since \(\sum a_i=4\equiv 1\pmod 3\) for
\[
(2,1,1),\quad (2,2,0),\quad (3,1,0),\quad (4,0,0),
\]
all four types belong to the same deformation component of projective bundles \(\mathbb{P}(\mathcal{E})\to\mathbb{P}^1\). The splitting types can be arranged, from more general to more special as follows:
\[
(2,1,1)\ \succ\ (2,2,0)\ \succ\ (3,1,0)\ \succ\ (4,0,0).
\]

\subsection{Infinitesimal deformations of conic bundles}

We study infinitesimal deformations of conic bundles, and show that for types $(2,2,0)$, $(3,1,0)$ and $(4,0,0)$ there exist embedded deformations of the conic bundle inside its ambient projective bundle that change the splitting type of the rank--$3$ bundle.

Fix the base curve $S:=\PP^{1}_{\kk}$ and a rank--$3$ vector
bundle $E$ on $S$. We denote by
\[
\pi\colon X:=\PP(E)\longrightarrow S
\]
the associated $\PP^{2}$--bundle. Let $i\colon Q\hookrightarrow X$ be a
smooth closed subscheme, flat over $S$, and set
$f:=\pi\circ i\colon Q\to S$.

We consider the following deformation functors on the category of Artin
local $\kk$--algebras with residue field $\kk$:
\begin{itemize}
\item[-] $\Def(E)$, the deformation functor of the vector bundle $E$ on $S$;
\item[-] $\Def(X/S)$, the deformation functor of $X$ as a $\PP^{2}$--bundle
      over $S$;
\item[-] $\Def(Q\subset X/S)$, the functor of simultaneous deformations of
      the pair $(Q\subset X)$ over $S$;
\item[-] $\Def(Q\subset X/S\mid X)$, the functor of embedded deformations
      of $Q$ inside the fixed ambient $X$ over $S$.
\end{itemize}
For any such functor $\mathcal{F}$ we write its Zariski tangent space as
\[
T\mathcal{F}:=\mathcal{F}\bigl(\kk[\varepsilon]/(\varepsilon^{2})\bigr).
\]
Recall that
\stepcounter{thm}
\begin{equation}\label{eq:given-isos}
T\Def(E)\;\cong\;H^{1}(S,\End E),
\qquad
T\Def(Q\subset X/S\mid X)\;\cong\;H^{0}(Q,N_{Q/X}),
\end{equation}
where $N_{Q/X}$ denotes the normal bundle of $Q$ in $X$.

\begin{Proposition}\label{prop:pair-conic-bundle}
In the notation above, there exists a canonical exact sequence
\stepcounter{thm}
\begin{equation}\label{eq:main-exact-sequence}
0 \longrightarrow H^{0}(Q,N_{Q/X})
  \longrightarrow T\Def(Q\subset X/S)
  \xrightarrow{\ \Phi\ } H^{1}(S,\End E)
  \xrightarrow{\ \delta\ } H^{1}(Q,N_{Q/X}),
\end{equation}
where:
\begin{itemize}
\item[-] $\Phi$ is induced by the natural morphism of deformation functors
      $\Def(Q\subset X/S)\to\Def(E)$ which forgets $Q$ and remembers
      only the deformation of $E$;
\item[-] $\delta$ is the composition
\[
\begin{aligned}
H^{1}(S,\End E)
&\;\cong\;T\Def(E)
\;\cong\;T\Def(X/S)
\longrightarrow \Ext^{1}\bigl(\Omega^{1}_{X/S},\OO_{X}\bigr)
\\
&\longrightarrow \Ext^{1}\bigl(\Omega^{1}_{X/S}\big|_{Q},\OO_{Q}\bigr)
\longrightarrow \Ext^{1}\bigl(\I/\I^{2},\OO_{Q}\bigr)
\cong H^{1}(Q,N_{Q/X}),
\end{aligned}
\]
where $\I\subset\OO_{X}$ is the ideal sheaf of $Q$, the last
identification uses $N_{Q/X}\cong(\I/\I^{2})^{\vee}$, and the middle
arrow is the connecting homomorphism associated with the relative
conormal exact sequence.
\end{itemize}

In particular, a first--order deformation of $E$ (equivalently, of the
$\PP^{2}$--bundle $X=\PP(E)$ over $S$) extends to a first--order
deformation of the pair $(Q\subset X)$ if and only if its image under
$\delta$ in $H^{1}(Q,N_{Q/X})$ vanishes.
\end{Proposition}
\begin{proof}
There is a natural forgetful morphism of deformation functors
\[
F\colon\Def(Q\subset X/S)\longrightarrow\Def(E),
\]
sending a deformation of the pair
\[
(Q_{A}\subset X_{A}=\PP(E_{A})\xrightarrow{\pi_{A}}S_{A})
\]
(over an Artin $\kk$--algebra $A$) to the underlying deformation
$E_{A}$ of $E$. Passing to tangent spaces for
$A=\kk[\varepsilon]/(\varepsilon^{2})$ we obtain a linear map
\[
dF\colon T\Def(Q\subset X/S)\longrightarrow T\Def(E)
\cong H^{1}(S,\End E),
\]
where we used the first isomorphism in \eqref{eq:given-isos}. By
definition this is the map $\Phi$ in
\eqref{eq:main-exact-sequence}.

An element of $\ker(dF)$ is a
first--order deformation of the pair $(Q\subset X)$ for which the
induced deformation of $E$ is trivial, i.e.\ isomorphic to the product
deformation $E\otimes_{\kk}\kk[\varepsilon]/(\varepsilon^{2})$. Since deformations of $E$ and deformations of
the projective bundle $\PP(E)$ over $S$ are equivalent, this means
that the deformation $X_{A}$ of $X$ is isomorphic to the trivial
deformation $X\times\Spec\kk[\varepsilon]/(\varepsilon^{2})$.

Fix such an isomorphism. Then any element of the kernel can be viewed as
a deformation of $Q$ as a closed subscheme of the fixed ambient
$X$, i.e.\ as an element of $T\Def(Q\subset X/S\mid X)$. Different
choices of the trivialization of the ambient deformation differ by an
automorphism of the trivial deformation which is the identity on the
special fiber, hence they induce isomorphic deformations of $Q$ inside
$X$. Therefore there is a canonical identification
\[
\ker(dF)\;\cong\;T\Def(Q\subset X/S\mid X).
\]
Using the second isomorphism in \eqref{eq:given-isos} we obtain
\[
\ker(dF)\;\cong\;H^{0}(Q,N_{Q/X}),
\]
and hence a short exact sequence
\[
0\longrightarrow H^{0}(Q,N_{Q/X})
  \longrightarrow T\Def(Q\subset X/S)
  \xrightarrow{\ \Phi\ } T\Def(E).
\]
Composing with the isomorphism
$T\Def(E)\cong H^{1}(S,\End E)$ gives the left exact part of
\eqref{eq:main-exact-sequence}:
\[
0\longrightarrow H^{0}(Q,N_{Q/X})
  \longrightarrow T\Def(Q\subset X/S)
  \xrightarrow{\ \Phi\ } H^{1}(S,\End E).
\]
Consider the smooth morphisms $\pi\colon X\to S$ and $f\colon Q\to S$.
The relative conormal sequence for the closed immersion $i\colon Q\to X$
is
\stepcounter{thm}
\begin{equation}\label{eq:relative-conormal}
0\longrightarrow \I/\I^{2}
 \longrightarrow \Omega^{1}_{X/S}\big|_{Q}
 \longrightarrow \Omega^{1}_{Q/S}
 \longrightarrow 0.
\end{equation}
Since $Q$ is smooth in $X$, the sheaf $\I/\I^{2}$ is locally free on
$Q$, and the normal bundle is
\[
N_{Q/X}:=(\I/\I^{2})^{\vee}.
\]
Apply the functor $\mathcal{H}om_{Q}(-,\OO_{Q})$ to
\eqref{eq:relative-conormal}. We get a long exact sequence
\[
\begin{aligned}
\cdots &\to
 \Hom\bigl(\Omega^{1}_{X/S}\big|_{Q},\OO_{Q}\bigr)
 \to \Hom\bigl(\I/\I^{2},\OO_{Q}\bigr)
 \\
 &\to \Ext^{1}\bigl(\Omega^{1}_{Q/S},\OO_{Q}\bigr)
 \xrightarrow{\ \psi\ }
 \Ext^{1}\bigl(\Omega^{1}_{X/S}\big|_{Q},\OO_{Q}\bigr)
 \xrightarrow{\ \partial\ }
 \Ext^{1}\bigl(\I/\I^{2},\OO_{Q}\bigr)
 \to \cdots
\end{aligned}
\]

Because $\I/\I^{2}$ is locally free and $N_{Q/X}=(\I/\I^{2})^{\vee}$, we
have
\[
\Ext^{i}\bigl(\I/\I^{2},\OO_{Q}\bigr)
\simeq H^{i}(Q,N_{Q/X}),\qquad i\ge 0.
\]
In particular,
\[
\Ext^{1}\bigl(\I/\I^{2},\OO_{Q}\bigr)\;\cong\;H^{1}(Q,N_{Q/X}).
\]

The map on the right in \eqref{eq:main-exact-sequence},
\[
\delta\colon H^{1}(S,\End E)\longrightarrow H^{1}(Q,N_{Q/X}),
\]
is defined as follows. By \eqref{eq:given-isos} and the standard
identification of first--order deformations of a smooth morphism with
$\Ext^{1}$ of its relative cotangent sheaf, we have
\[
H^{1}(S,\End E)\cong T\Def(E)\cong T\Def(X/S)
\cong \Ext^{1}\bigl(\Omega^{1}_{X/S},\OO_{X}\bigr).
\]
Restricting to $Q$ gives a homomorphism
\[
\Ext^{1}\bigl(\Omega^{1}_{X/S},\OO_{X}\bigr)
\longrightarrow \Ext^{1}\bigl(\Omega^{1}_{X/S}\big|_{Q},\OO_{Q}\bigr),
\]
and composing with the connecting map
\[
\partial\colon
\Ext^{1}\bigl(\Omega^{1}_{X/S}\big|_{Q},\OO_{Q}\bigr)
\longrightarrow
\Ext^{1}\bigl(\I/\I^{2},\OO_{Q}\bigr)
\cong H^{1}(Q,N_{Q/X})
\]
we obtain $\delta$. We must show that
\[
\Ima(\Phi)=\ker(\delta)\subset H^{1}(S,\End E).
\]

Let $\xi\in T\Def(Q\subset X/S)$ be a first--order deformation of the
pair, and let $\eta=\Phi(\xi)\in H^{1}(S,\End E)$ be the induced
first--order deformation of $E$ (equivalently, of $X$ over $S$). By
construction, $\xi$ consists of a deformation $X_{A}$ of $X$ over $S_{A}$,
together with a deformation $Q_{A}\subset X_{A}$ of $Q$. The class of
$X_{A}$ determines an element of
$\Ext^{1}(\Omega^{1}_{X/S},\OO_{X})$, whose restriction to $Q$ is an
element of $\Ext^{1}(\Omega^{1}_{X/S}|_{Q},\OO_{Q})$. The existence of
$Q_{A}\subset X_{A}$ precisely means that this element lies in the image
of
\[
\psi\colon
\Ext^{1}\bigl(\Omega^{1}_{Q/S},\OO_{Q}\bigr)\longrightarrow
\Ext^{1}\bigl(\Omega^{1}_{X/S}\big|_{Q},\OO_{Q}\bigr)
\]
in the long exact sequence above. By exactness at
$\Ext^{1}(\Omega^{1}_{X/S}|_{Q},\OO_{Q})$, this is equivalent to
$\partial(\operatorname{res}(\eta))=0$, hence to $\delta(\eta)=0$ in
$H^{1}(Q,N_{Q/X})$. Thus $\Ima(\Phi)\subset\ker(\delta)$.

Conversely, let $\eta\in H^{1}(S,\End E)$ be a first--order deformation
of $E$ and consider the associated deformation $X_{A}$ of $X$ over
$S_{A}$. Suppose that $\delta(\eta)=0$. By definition of $\delta$, this
means that the element
\[
\operatorname{res}(\eta)\in
\Ext^{1}\bigl(\Omega^{1}_{X/S}\big|_{Q},\OO_{Q}\bigr)
\]
coming from the deformation of $X$ lies in the image of $\psi$. Hence
there exists a first--order deformation $Q_{A}\to S_{A}$ of $Q$ whose
class in $\Ext^{1}(\Omega^{1}_{Q/S},\OO_{Q})$ maps to
$\operatorname{res}(\eta)$. The equality of these classes expresses
precisely the compatibility condition for $Q_{A}$ to embed as a closed
subscheme $Q_{A}\subset X_{A}$ deforming $Q\subset X$. Thus $\eta$ comes
from a deformation of the pair $(Q\subset X)$, i.e.\ there exists
$\xi\in T\Def(Q\subset X/S)$ with $\Phi(\xi)=\eta$. Hence
$\ker(\delta)\subset\Ima(\Phi)$.
\end{proof}

\begin{Corollary}\label{cor:conic-bundle}
In the situation of Proposition~\ref{prop:pair-conic-bundle}, assume in
addition that $Q\subset X=\PP(E)$ is a smooth conic bundle. The normal bundle of $Q$ in $X$ is
\[
N_{Q/X}\;\cong\;\OO_{Q}(Q)
\;\cong\;\bigl(\OO_{X}(2)\otimes\pi^{*}\mathcal{L}\bigr)\big|_{Q}
\;\cong\;\OO_{Q}(2)\otimes f^{*}\mathcal{L},
\]
and the exact sequence of
Proposition~\ref{prop:pair-conic-bundle} takes the form
$$
0 \longrightarrow
H^{0}\bigl(Q,\OO_{Q}(2)\otimes f^{*}\mathcal{L}\bigr)
\longrightarrow
T\Def(Q\subset\PP(E)/S)
\xrightarrow{\ \Phi\ }
H^{1}\bigl(S,\End E\bigr)
\xrightarrow{\ \delta\ }
H^{1}\bigl(Q,\OO_{Q}(2)\otimes f^{*}\mathcal{L}\bigr).
$$

In particular, a first--order deformation of the vector bundle $E$
(equivalently, of the $\PP^{2}$--bundle $X=\PP(E)$ over $S$) extends to
a first--order deformation of the conic bundle $Q\subset\PP(E)$ if and
only if its image under
\[
\delta\colon H^{1}(S,\End E)\longrightarrow
H^{1}\bigl(Q,\OO_{Q}(2)\otimes f^{*}\mathcal{L}\bigr)
\]
vanishes.
\end{Corollary}

\begin{proof}
By assumption, $Q$ is the zero locus of a regular section of the line
bundle
\[
L := \OO_{X}(2)\otimes\pi^{*}\mathcal{L}.
\]
Thus $Q$ is an effective Cartier divisor on $X$, and its divisor class
is
\[
[Q] = c_{1}(L) = 2\xi+\pi^{*}c_{1}(\mathcal{L}),
\]
where $\xi=c_{1}\bigl(\OO_{X}(1)\bigr)$. Since $Q$ is smooth in $X$, its
normal bundle is the restriction of $\OO_{X}(Q)$:
\[
N_{Q/X}\;\cong\;\OO_{X}(Q)\big|_{Q}.
\]
On the other hand, the line bundle associated with the Cartier divisor
$Q$ is precisely $L$, so $\OO_{X}(Q)\cong L$. Therefore
\[
N_{Q/X} \cong \OO_{X}(Q)\big|_{Q}
          \cong L\big|_{Q}
          \cong \bigl(\OO_{X}(2)\otimes\pi^{*}\mathcal{L}\bigr)\big|_{Q}
          \cong \OO_{Q}(2)\otimes f^{*}\mathcal{L},
\]
where we used $f=\pi\circ i\colon Q\to S$ and the identification
$i^{*}\OO_{X}(1)=\OO_{Q}(1)$.

Substituting $N_{Q/X}\cong\OO_{Q}(2)\otimes f^{*}\mathcal{L}$ into the exact
sequence of Proposition~\ref{prop:pair-conic-bundle} we get the claim.
\end{proof}

\begin{Corollary}\label{cor:splitting-deformation}
Let $\pi\colon \mathcal{Q}^1\to \PP^1$ be a splitting conic bundle with
\[
\mathcal{Q}^1 \;\subset\;
\mathcal{T}_{a_0,a_1,a_2}\;\cong\;\PP\big(\Ecal_{a_0,a_1,a_2}\big),
\qquad
\Ecal_{a_0,a_1,a_2}\;\cong\;
\OO_{\PP^1}(a_0)\oplus \OO_{\PP^1}(a_1)\oplus \OO_{\PP^1}(a_2),
\]
of multidegree $(d_{0,0},d_{0,1},d_{0,2},d_{1,1},d_{1,2},d_{2,2})$ satisfying 
\[
m\;\coloneqq\;d_{0,0}-2a_0
=d_{0,1}-a_0-a_1
=d_{0,2}-a_0-a_2
=d_{1,1}-2a_1
=d_{1,2}-a_1-a_2
=d_{2,2}-2a_2.
\]
Set $\mathcal{L}:=\OO_{\PP^1}(m)$. Then $\mathcal{Q}^1$ is a divisor in $\PP(\Ecal_{a_0,a_1,a_2})$ cut out by a section of
$\OO_{\PP(\Ecal_{a_0,a_1,a_2})}(2)\otimes \pi^{*}\mathcal{L}$, the normal bundle is
\[
N_{\mathcal{Q}^1/\PP(\Ecal_{a_0,a_1,a_2})}
\;\cong\;
\OO_{\mathcal{Q}^1}(2)\otimes \pi^{*}\mathcal{L},
\]
and there is a canonical exact sequence
\[
\begin{aligned}
0 \longrightarrow\;
&H^{0}\bigl(\mathcal{Q}^1,\OO_{\mathcal{Q}^1}(2)\otimes \pi^{*}\mathcal{L}\bigr)
\longrightarrow
T\Def\bigl(\mathcal{Q}^1\subset \PP(\Ecal_{a_0,a_1,a_2})/\PP^1\bigr)
\\
\longrightarrow\;
&H^{1}\bigl(\PP^1,\End \Ecal_{a_0,a_1,a_2}\bigr)
\longrightarrow
H^{1}\bigl(\mathcal{Q}^1,\OO_{\mathcal{Q}^1}(2)\otimes \pi^{*}\mathcal{L}\bigr),
\end{aligned}
\]
where
\[
\End \Ecal_{a_0,a_1,a_2}
\;\cong\;
\bigoplus_{0\le i,j\le 2}\OO_{\PP^1}(a_i-a_j).
\]
Equivalently
\[
\begin{aligned}
0 \longrightarrow\;
&H^{0}\bigl(\mathcal{Q}^1,\OO_{\mathcal{Q}^1}(2)\otimes \pi^{*}\mathcal{L}\bigr)
\longrightarrow
T\Def\bigl(\mathcal{Q}^1\subset \PP(\Ecal_{a_0,a_1,a_2})/\PP^1\bigr)
\\
\longrightarrow\;
&\bigoplus_{0\le i,j\le 2} H^{1}\bigl(\PP^1,\OO_{\PP^1}(a_i-a_j)\bigr)
\longrightarrow
H^{1}\bigl(\mathcal{Q}^1,\OO_{\mathcal{Q}^1}(2)\otimes \pi^{*}\mathcal{L}\bigr).
\end{aligned}
\]
In particular, a first--order deformation of the underlying vector bundle $\Ecal_{a_0,a_1,a_2}$
is induced by an embedded deformation of the splitting conic bundle
$\mathcal{Q}^1\subset \PP(\Ecal_{a_0,a_1,a_2})$ if and only if its class in
\[
H^{1}\bigl(\PP^1,\End \Ecal_{a_0,a_1,a_2}\bigr)
\;\cong\;
\bigoplus_{i,j} H^{1}\bigl(\PP^1,\OO_{\PP^1}(a_i-a_j)\bigr)
\]
lies in the kernel of the boundary map
\[
H^{1}\bigl(\PP^1,\End \Ecal_{a_0,a_1,a_2}\bigr)
\longrightarrow
H^{1}\bigl(\mathcal{Q}^1,\OO_{\mathcal{Q}^1}(2)\otimes \pi^{*}\mathcal{L}\bigr).
\]
\end{Corollary}

\begin{Remark}
We now apply Corollary~\ref{cor:splitting-deformation} to surface conic bundles
$\pi\colon S\to\PP^1$ with discriminant of degree $d_S=8$. 

One checks that in all four cases 
\[
m=d_{0,0}-2a_0=d_{1,1}-2a_1=d_{2,2}-2a_2=0,
\]
so that $\mathcal{L}\simeq\OO_{\PP^1}$ and
\[
N_{S/\PP(\Ecal_{a_0,a_1,a_2})}\;\cong\;\OO_{S}(2).
\]
Thus Corollary~\ref{cor:splitting-deformation} specializes to
\[
0 \longrightarrow H^{0}\bigl(S,\OO_{S}(2)\bigr)
\longrightarrow T\Def\bigl(S\subset \PP(\Ecal_{a_0,a_1,a_2})/\PP^1\bigr)
\longrightarrow H^{1}\bigl(\PP^1,\End \Ecal_{a_0,a_1,a_2}\bigr)
\longrightarrow H^{1}\bigl(S,\OO_{S}(2)\bigr).
\]

\smallskip\noindent
\emph{Cohomology of the normal bundle.}
We first compute $H^{i}(S,\OO_S(2))$, hence the cohomology of the normal bundle
$N_{S/\PP(\Ecal_{a_0,a_1,a_2})}\cong\OO_S(2)$, in terms of the splitting type
\[
\Ecal_{a_0,a_1,a_2}\cong\OO_{\PP^1}(a_0)\oplus\OO_{\PP^1}(a_1)\oplus\OO_{\PP^1}(a_2),
\qquad a_0\ge a_1\ge a_2\ge 0.
\]
Let $X=\PP(\Ecal_{a_0,a_1,a_2})$ and let $H$ be the relative hyperplane class, so
$\OO_X(1)=\OO_X(H)$ and $S$ is linearly equivalent to $2H$. Since $S$ is the zero locus of a
section of $\OO_X(2)$, there is an exact sequence
\stepcounter{thm}
\begin{equation}\label{eq:OX-OS2}
0\longrightarrow\OO_X\longrightarrow\OO_X(2)\longrightarrow\OO_S(2)\longrightarrow 0.
\end{equation}
The projection $\pi\colon X\to\PP^1$ is a $\PP^{2}$--bundle, and for $d\ge 0$ one has
\[
\pi_{*}\OO_X(d)\;\cong\;\Sym^{d}\Ecal_{a_0,a_1,a_2},
\qquad R^{q}\pi_{*}\OO_X(d)=0\ \text{for }q>0.
\]
In particular
\[
\pi_{*}\OO_X\cong\OO_{\PP^1},\ \ R^{q}\pi_{*}\OO_X=0\ (q>0),\qquad
\pi_{*}\OO_X(2)\cong\Sym^{2}\Ecal_{a_0,a_1,a_2},\ \ R^{q}\pi_{*}\OO_X(2)=0\ (q>0).
\]
Pushing \eqref{eq:OX-OS2} down to $\PP^1$ and using $R^{1}\pi_{*}\OO_S(2)=0$
(fibers of $S\to\PP^1$ are conics and $H^{1}(\text{conic},\OO(2))=0$) we obtain a short exact
sequence of vector bundles on $\PP^1$
\[
0\longrightarrow\OO_{\PP^1}\longrightarrow\Sym^{2}\Ecal_{a_0,a_1,a_2}
\longrightarrow\pi_{*}\OO_S(2)\longrightarrow 0,
\]
and Leray’s spectral sequence yields
\[
H^{0}(S,\OO_S(2))\;\cong\;H^{0}\!\bigl(\PP^1,\pi_{*}\OO_S(2)\bigr),\qquad
H^{1}(S,\OO_S(2))\;\cong\;H^{1}\!\bigl(\PP^1,\pi_{*}\OO_S(2)\bigr).
\]
The associated long exact sequence in cohomology gives
\[
0\longrightarrow H^{0}\bigl(\PP^1,\OO_{\PP^1}\bigr)
\longrightarrow H^{0}\bigl(\PP^1,\Sym^{2}\Ecal_{a_0,a_1,a_2}\bigr)
\longrightarrow H^{0}\bigl(S,\OO_S(2)\bigr)\longrightarrow 0,
\]
since $H^{1}(\PP^1,\OO_{\PP^1})=0$, and
\[
0\longrightarrow H^{1}\bigl(\PP^1,\Sym^{2}\Ecal_{a_0,a_1,a_2}\bigr)
\longrightarrow H^{1}\bigl(S,\OO_S(2)\bigr)\longrightarrow 0.
\]
Thus
$$
H^{1}\bigl(S,\OO_S(2)\bigr)\;\cong\;H^{1}\bigl(\PP^1,\Sym^{2}\Ecal_{a_0,a_1,a_2}\bigr),
\qquad
h^{0}\bigl(S,\OO_S(2)\bigr)=h^{0}\bigl(\PP^1,\Sym^{2}\Ecal_{a_0,a_1,a_2}\bigr)-1.
$$
Since
\[
\Ecal_{a_0,a_1,a_2}\cong\OO(a_0)\oplus\OO(a_1)\oplus\OO(a_2),
\]
we have
\[
\Sym^{2}\Ecal_{a_0,a_1,a_2}\;\cong\;
\bigoplus_{0\le i\le j\le 2}\OO_{\PP^1}(a_i+a_j),
\]
and on $\PP^1$ we know
\[
H^{1}\bigl(\PP^1,\OO_{\PP^1}(d)\bigr)=0\ \text{for }d\ge -1.
\]
In our four types all $a_i\ge 0$, hence $a_i+a_j\ge 0$ for all $i,j$, and
\[
H^{1}\bigl(\PP^1,\Sym^{2}\Ecal_{a_0,a_1,a_2}\bigr)=0,
\qquad
H^{1}\bigl(S,\OO_S(2)\bigr)=0.
\]
Moreover
\[
h^{0}\bigl(\PP^1,\Sym^{2}\Ecal_{a_0,a_1,a_2}\bigr)
=\sum_{0\le i\le j\le 2}h^{0}\bigl(\PP^1,\OO_{\PP^1}(a_i+a_j)\bigr)
=\sum_{0\le i\le j\le 2}(a_i+a_j+1),
\]
so
\[
h^{0}\bigl(S,\OO_S(2)\bigr)
=\sum_{0\le i\le j\le 2}(a_i+a_j+1)-1.
\]
A direct computation shows that for each of the four $d_S=8$ splitting types
\[
(a_0,a_1,a_2)\in\{(2,1,1),(2,2,0),(3,1,0),(4,0,0)\}
\]
one obtains
\[
h^{1}\bigl(S,\OO_S(2)\bigr)=0,\qquad h^{0}\bigl(S,\OO_S(2)\bigr)=21.
\]
In particular, for our conic bundles with $d_S=8$ the last arrow in the exact sequence above
\[
H^{1}\bigl(\PP^1,\End \Ecal_{a_0,a_1,a_2}\bigr)
\longrightarrow H^{1}\bigl(S,\OO_{S}(2)\bigr)
\]
is the zero map, and we have a short exact sequence
\stepcounter{thm}
\begin{equation}\label{eq:short-exact-d8}
0 \longrightarrow H^{0}\bigl(S,\OO_{S}(2)\bigr)
\longrightarrow T\Def\bigl(S\subset \PP(\Ecal_{a_0,a_1,a_2})/\PP^1\bigr)
\longrightarrow H^{1}\bigl(\PP^1,\End \Ecal_{a_0,a_1,a_2}\bigr)
\longrightarrow 0.
\end{equation}

\smallskip
Since
\[
\End \Ecal_{a_0,a_1,a_2}
\;\cong\;
\bigoplus_{0\le i,j\le 2}\OO_{\PP^1}(a_i-a_j),
\]
we can compute $H^{1}(\PP^1,\End \Ecal_{a_0,a_1,a_2})$ explicitly using
\[
H^{1}\bigl(\PP^1,\OO_{\PP^1}(t)\bigr)\cong 0 \text{ for } t\ge -1 \text{ and }
\dim H^{1}\bigl(\PP^1,\OO_{\PP^1}(-d)\bigr)=d-1 \text{ for } d\ge 2.
\]
Writing $a_0\ge a_1\ge a_2\ge 0$, we obtain:

\medskip\noindent
\emph{Type $(2,1,1)$:} here
\[
(a_0,a_1,a_2)=(2,1,1),
\]
so the degree matrix $(a_i-a_j)$ is
\[
\begin{pmatrix}
0 & 1 & 1\\
-1 & 0 & 0\\
-1 & 0 & 0
\end{pmatrix}.
\]
All negative degrees are $-1$, hence $H^{1}(\PP^1,\OO_{\PP^1}(a_i-a_j))=0$ for every
pair $(i,j)$, and
\[
H^{1}\bigl(\PP^1,\End \Ecal_{2,1,1}\bigr)=0.
\]
In particular, there are no first--order embedded
deformations of $S\subset \PP(\Ecal_{2,1,1})$ that change the splitting type of
$\Ecal_{2,1,1}$.

\medskip\noindent
\emph{Type $(2,2,0)$:} here
\[
(a_0,a_1,a_2)=(2,2,0),
\]
and
\[
(a_i-a_j)=
\begin{pmatrix}
0 & 0 & 2\\
0 & 0 & 2\\
-2 & -2 & 0
\end{pmatrix}.
\]
The only negative entries are $-2$, occurring twice; hence
\[
\dim H^{1}\bigl(\PP^1,\End \Ecal_{2,2,0}\bigr)
=2\cdot\dim H^{1}\bigl(\PP^1,\OO_{\PP^1}(-2)\bigr)
=2.
\]
So abstractly the bundle $\Ecal_{2,2,0}$ has a $2$--dimensional space of first--order
deformations (necessarily changing its splitting type). In view of
\eqref{eq:short-exact-d8}, every class in $H^{1}(\PP^1,\End \Ecal_{2,2,0})$ is induced by an embedded deformation of the pair
$S\subset \PP(\Ecal_{2,2,0})$.

\medskip\noindent
\emph{Type $(3,1,0)$:} here
\[
(a_0,a_1,a_2)=(3,1,0),
\qquad
(a_i-a_j)=
\begin{pmatrix}
0 & 2 & 3\\
-2 & 0 & 1\\
-3 & -1 & 0
\end{pmatrix}.
\]
The negative degrees $\le -2$ are $-2$ and $-3$, each appearing once, so
\[
\dim H^{1}\bigl(\PP^1,\End \Ecal_{3,1,0}\bigr)
= \dim H^{1}\bigl(\PP^1,\OO_{\PP^1}(-2)\bigr)
+ \dim H^{1}\bigl(\PP^1,\OO_{\PP^1}(-3)\bigr)
=1+2=3.
\]
Hence there is a $3$--dimensional space of infinitesimal deformations of
$\Ecal_{3,1,0}$ (again, necessarily changing its splitting type). By \eqref{eq:short-exact-d8} all these classes lift to embedded first--order deformations of the pair
$(S\subset\PP(\Ecal_{3,1,0}))$.

\medskip\noindent
\emph{Type $(4,0,0)$:} here
\[
(a_0,a_1,a_2)=(4,0,0),
\qquad
(a_i-a_j)=
\begin{pmatrix}
0 & 4 & 4\\
-4 & 0 & 0\\
-4 & 0 & 0
\end{pmatrix}.
\]
The negative degree $-4$ occurs twice, and
\[
\dim H^{1}\bigl(\PP^1,\OO_{\PP^1}(-4)\bigr)
=\dim H^{0}\bigl(\PP^1,\OO_{\PP^1}(2)\bigr)=3,
\]
so
\[
\dim H^{1}\bigl(\PP^1,\End \Ecal_{4,0,0}\bigr)=2\cdot 3=6.
\]
Thus the underlying bundle $\Ecal_{4,0,0}$ has a $6$--dimensional space of
first--order deformations changing its splitting. Again, by \eqref{eq:short-exact-d8}
every such class is realized by an embedded deformation of the conic bundle
$S\subset\PP(\Ecal_{4,0,0})$.

\medskip
Summarizing, for conic bundles $S\to\PP^1$ with discriminant degree $d_S=8$:
\begin{itemize}
\item[-] in type $(2,1,1)$ one has $H^{1}(\PP^1,\End \Ecal_{2,1,1})=0$, so no
      first--order embedded deformation of the pair $(S\subset\PP(\Ecal_{2,1,1}))$
      can change the splitting type of the rank--$3$ bundle;
\item[-] in types $(2,2,0)$, $(3,1,0)$, and $(4,0,0)$ the groups
      $H^{1}(\PP^1,\End \Ecal_{a_0,a_1,a_2})$ have dimensions $2$, $3$, and $6$,
      respectively. Our computation of the cohomology of the normal bundle shows that
      $H^{1}(S,\OO_S(2))=0$ in all four types, so the boundary map
      \[
      H^{1}\bigl(\PP^1,\End \Ecal_{a_0,a_1,a_2}\bigr)\longrightarrow H^{1}\bigl(S,\OO_{S}(2)\bigr)
      \]
      is identically zero. Consequently the exact sequence in
      Corollary~\ref{cor:splitting-deformation} reduces to
      \[
      0\to H^{0}(S,\OO_S(2))\to T\Def\bigl(S\subset\PP(\Ecal_{a_0,a_1,a_2})/\PP^1\bigr)
      \twoheadrightarrow H^{1}\bigl(\PP^1,\End \Ecal_{a_0,a_1,a_2}\bigr)\to 0,
      \]
      and every infinitesimal deformation of the vector bundle $\Ecal_{a_0,a_1,a_2}$
      is induced by an embedded first--order deformation of the pair $(S,\PP(\Ecal_{a_0,a_1,a_2}))$.
      Vector bundles on $\PP^1$ with fixed splitting
      type have no nontrivial deformations, so any nonzero class in
      $H^{1}(\PP^1,\End \Ecal_{a_0,a_1,a_2})$ corresponds to a deformation in which the
      splitting type actually changes. Thus, for types $(2,2,0)$, $(3,1,0)$ and $(4,0,0)$
      there exist embedded deformations of the conic bundle inside its ambient
      projective bundle that change the splitting type of the rank--$3$ bundle.
\end{itemize}
\end{Remark}

Next, we construct explicit $1$-parameter deformations of conic bundles changing the splitting type of the ambient projective bundle. 

\subsubsection{Degenerating $\mathcal{T}_{2,1,1}$ to $\mathcal{T}_{2,2,0}$}\label{(2,1,1)->(2,2,0)}

We work in the products
\[
\PP^9\times\PP^1\;\; \text{ with coordinates } \;\; [Z_0:\cdots:Z_9]\,,\ [A:B],
\qquad
\PP^6\times\PP^1\;\; \text{ with coordinates } \;\; [z_0:\cdots:z_6]\,,\ [a:b].
\]

Let $\widetilde{X}$ be the subscheme cut by the $2\times 2$ minors of the $2\times 7$ matrix
\[
M_{A,B}\;=\;
\begin{pmatrix}
	Z_0 & Z_1 & Z_2 & Z_4 & Z_5 & A\,Z_6+B\,Z_7 & Z_8\\[2pt]
	Z_1 & Z_2 & Z_3 & Z_5 & Z_6 & A\,Z_7+B\,Z_8 & Z_9
\end{pmatrix}.
\]
For $[A:B]\in\PP^1$ denote by
\[
X_{A,B}\;\coloneqq\;\widetilde{X}\cap\big(\PP^9\times\{[A:B]\}\big)\;\subset\;\PP^9
\]
the fiber over $[A:B]$.

\begin{Lemma}\label{lem:X-fibers}
	For $B\neq 0$ one has $X_{A,B}\cong \PP\!\big(\OO_{\PP^1}(3)\oplus\OO_{\PP^1}(2)\oplus\OO_{\PP^1}(2)\big)$ i.e. $X_{A,B}$ is the rational normal scroll $S(3,2,2)\subset\PP^9$. For $[A:B]=[1:0]$ one has $X_{1,0}\cong \PP\!\big(\OO_{\PP^1}(3)\oplus\OO_{\PP^1}(3)\oplus\OO_{\PP^1}(1)\big)=S(3,3,1)\subset\PP^9$.
\end{Lemma}

\begin{proof}
	The ideal of a $3$–fold rational normal scroll $S(\alpha,\beta,\gamma)\subset\PP^{\alpha+\beta+\gamma+2}$ is generated by the $2\times 2$ minors of a $2\times(\alpha+\beta+\gamma+3)$ matrix of linear forms whose columns come in three consecutive blocks of lengths $\alpha,\beta,\gamma$.
	
	Assume $B\neq 0$. Up to scaling we may take $B=1$ and keep $A$ as a parameter. Set new linear coordinates
	\[
	W_6\coloneqq A\,Z_6+Z_7,\qquad W_7\coloneqq A\,Z_7+Z_8.
	\]
	Then $M_{A,1}$ is equivalent, via invertible linear changes in the ambient $\PP^9$, to the matrix 
	\[
	(Z_0,Z_1),(Z_1,Z_2),(Z_2,Z_3)\;\big|\; (Z_4,Z_5),(Z_5,W_6)\;\big|\; (W_7,Z_8),
	\]
	i.e.\ to the standard $S(3,2,2)$. In particular, the $2\times 2$ minors of $M_{A,1}$ vanish on the image of the morphism
	\[
	\begin{footnotesize}
		\begin{array}{lccc}
			F: & \mathcal{T}_{2,1,1} & \longrightarrow  & \PP^9 \\ 
			& [x_0:x_1;\,y_0:y_1:y_2] & \longmapsto & [\,y_0x_0^3:\ y_0x_0^2x_1:\ y_0x_0x_1^2:\ y_0x_1^3:\ y_1x_0^2:\ y_1x_0x_1:\ y_1x_1^2:\ y_2x_0^2:\ y_2x_0x_1:\ y_2x_1^2\,]. 
		\end{array} 
	\end{footnotesize}
	\]
	whose image is $\PP(\OO(3)\oplus\OO(2)\oplus\OO(2))$. 
	
	Take $B = 0$. Now, the mixed column becomes $(Z_6;Z_7)$, so the seven consecutive pairs read
	\[
	(Z_0,Z_1),(Z_1,Z_2),(Z_2,Z_3)\;\big|\; (Z_4,Z_5),(Z_5,Z_6),(Z_6,Z_7)\;\big|\; (Z_8,Z_9),
	\]
	i.e.\ two blocks of lengths $3$ and $3$ and a block of length $1$, the standard presentation of $S(3,3,1)\subset\PP^9$. Equivalently, $X_{1,0}$ is the image of
	\[
	\begin{footnotesize}
		\begin{array}{lccc}
			G: & \mathcal{T}_{2,2,0} & \longrightarrow  & \PP^9 \\ 
			& [x_0:x_1;\,y_0:y_1:y_2] & \longmapsto & [\,y_0x_0^3:y_0x_0^2x_1:y_0x_0x_1^2:y_0x_1^3:\; y_1x_0^3:y_1x_0^2x_1:y_1x_0x_1^2:y_1x_1^3:\; y_2x_0:y_2x_1\,]
		\end{array} 
	\end{footnotesize}
	\]
	completing the proof.
\end{proof}

Let $\widetilde{Y}$ be the subscheme cut by the $2\times 2$ minors of
\[
N_{a,b}\;=\;
\begin{pmatrix}
	z_0 & z_1 & z_3 & a^{2}\,z_4+b\,z_5\\[2pt]
	z_1 & z_2 & z_4 & a\,z_5+b^{2}\,z_6
\end{pmatrix}.
\]
For $[a:b]\in\PP^1$ denote by
\(
Y_{a,b}\coloneqq \widetilde{Y}\cap\big(\PP^6\times\{[a:b]\}\big)\subset\PP^6
\)
the fiber.

\begin{Lemma}\label{lem:Y-fibers}
	For $b\neq 0$ one has $Y_{a,b}\cong \PP\!\big(\OO_{\PP^1}(2)\oplus\OO_{\PP^1}(1)\oplus\OO_{\PP^1}(1)\big)$ i.e. $Y_{a,b}=S(2,1,1)\subset\PP^6$. Moreover $Y_{1,0}$ is the cone of vertex $[0:\cdots:0:1]$ over the surface scroll $S(2,2)\subset\{z_6=0\}\cong\PP^5$.
\end{Lemma}

\begin{proof}
	If $b\neq 0$, scale to $b=1$ and perform the linear change
	\[
	w_4\coloneqq a^{2}z_4+z_5,\qquad w_5\coloneqq a z_5+z_6.
	\]
	Then $N_{a,1}$ becomes a matrix with blocks of lengths $2,1,1$, i.e. the standard $S(2,1,1)$. Equivalently, $Y_{a,1}$ is the image of
	\[
	\begin{footnotesize}
		\begin{array}{lccc}
			f: & \mathcal{T}_{2,1,1} & \longrightarrow  & \PP^6 \\ 
			& [x_0:x_1;\,y_0:y_1:y_2] & \longmapsto & [\,x_0^2y_0:x_0x_1y_0:x_1^2y_0:\; x_0y_1:x_1y_1:\; x_0y_2:x_1y_2\,].
		\end{array} 
	\end{footnotesize}
	\]
	If $[a:b]=[1:0]$ then
	$$
	N_{1,0}=\begin{pmatrix} z_0&z_1&z_3&z_4\\ z_1&z_2&z_4&z_5\end{pmatrix}
	$$
	and $z_6$ is free, so $Y_{1,0}$ is the cone over $S(2,2)\subset\PP^5$ with vertex $[0:\cdots:0:1]$, or equivalently the image of the map 
	\stepcounter{thm}
	\begin{equation}\label{morg}
		\begin{footnotesize}
			\begin{array}{lccc}
				g: & \mathcal{T}_{2,2,0} & \longrightarrow  & \PP^6 \\ 
				& [x_0:x_1;\,y_0:y_1:y_2] & \longmapsto & [\,x_0^2y_0:x_0x_1y_0:x_1^2y_0:\; x_0^2y_1:x_0x_1y_1:x_1^2y_1:\; y_2\,]
			\end{array} 
		\end{footnotesize}
	\end{equation}
	contracting $\{y_0 = y_1 = 0\}$ to the point $[0:\dots :0:1]$.
\end{proof}

Inside $\PP^9\times\PP^1$ consider the family of linear subspaces
\[
H_{A,B}\;=\;\{Z_1=Z_2=Z_3=Z_5=Z_6=A\,Z_7+B\,Z_8=Z_9=0\}.
\]
These yield a rational map
\[
\begin{small}
	\begin{array}{lccc}
		\mathrm{pr}: & \PP^9\times\PP^1 & \dashrightarrow  & \PP^6\times\PP^1 
	\end{array} 
\end{small}
\]
i.e a fiberwise linear projection from the center $H_{A,B}$. Restricting $\mathrm{pr}$ to the fibers of $\widetilde{X}\to\PP^1$ we obtain rational maps
\[
\mathrm{pr}_{A,B}:\ X_{A,B}\dashrightarrow Y_{A,B}.
\]
On the charts $\mathcal{U}_{A,1} = \{Z_9 = 1, B = 1\}$ of $\PP^9\times\PP^1$ and $\mathcal{V}_{A,1} = \{z_6 = 1, b = 1\}$ of $\PP^6\times\PP^1$ the map $\mathrm{pr}$ is given by 
\[
\begin{small}
	\begin{array}{lccc}
		\mathrm{pr}_{\mathcal{U}_{A,1}}: & \mathcal{U}_{A,1} & \dashrightarrow  & \mathcal{V}_{A,1} \\ 
		& ([Z],[A:1]) & \longmapsto & \big(Z_0,Z_4,Z_8,Z_5,Z_6,A\,Z_7+Z_8,A\big),
	\end{array} 
\end{small}
\]
and on the charts $\mathcal{U}_{1,B} = \{Z_9 = 1, A = 1\}$ of $\PP^9\times\PP^1$ and $\mathcal{V}_{1,B} = \{z_6 = 1, a = 1\}$ of $\PP^6\times\PP^1$ the map $\mathrm{pr}$ is given by 
\[
\begin{small}
	\begin{array}{lccc}
		\mathrm{pr}_{\mathcal{U}_{1,B}}: & \mathcal{U}_{1,B} & \dashrightarrow  & \mathcal{V}_{1,B} \\ 
		& ([Z],[1:B]) & \longmapsto & \big(Z_0,Z_4,Z_8,Z_5,Z_6,Z_7+B\,Z_8,Z_9,B\big).
	\end{array} 
\end{small}
\]

\begin{Proposition}\label{prop:pr-birational}
	For $B\neq 0$ the restriction $\mathrm{pr}_{A,B}$ yields an isomorphism $X_{A,B}\xrightarrow{\;\sim\;}Y_{A,B}$. For $[A:B]=[1:0]$ the map $\mathrm{pr}_{1,0}$ contracts
	$$
	\ell=\{y_0=y_1=0\}\cong\PP^1\subset X_{1,0}\cong S(3,3,1)
	$$
	to the vertex $[0:\cdots:0:1]$ of the cone $Y_{1,0}$, and is an isomorphism away from $\ell$. In particular, the rational map
	\[
	\pi\;:\;\widetilde{X}\dashrightarrow \widetilde{Y}
	\]
	induced by $\mathrm{pr}_{A,B}$ is a birational morphism; it is an isomorphism over $\{B\neq 0\}$ and contracts a single $\PP^1$ inside $X_{1,0}$.
\end{Proposition}
\begin{proof}
	One checks directly that
	\[
	\mathrm{pr}_{0,1}\circ F\,=\,f\qquad\text{and}\qquad
	\mathrm{pr}_{1,0}\circ G\,=\,g,
	\]
	where $F,f,G,g$ are the maps in the proofs of Lemma~\ref{lem:X-fibers} and Lemma~\ref{lem:Y-fibers}. The same holds for any $B\neq 0$ after the linear changes used in those proofs.
\end{proof}

Let us describe a divisor $\widetilde{\mathcal{C}}\subset\widetilde{Y}$ whose fibers are surface conic bundles, and its strict transform inside $\widetilde{X}$.

In $\Tcal_{2,1,1}=\PP\!\big(\OO_{\PP^1}(2)\oplus\OO_{\PP^1}(1)\oplus\OO_{\PP^1}(1)\big)$ consider the splitting conic bundle
\[
F_{2,1,1}\;=\; s_{00}\,y_0^2+s_{01}\,y_0y_1+s_{02}\,y_0y_2+s_{11}\,y_1^2+s_{12}\,y_1y_2+s_{22}\,y_2^2,
\]
with coefficients
\begin{align*}
	s_{00}&= a_0 x_0^4+a_1 x_0^3x_1+a_2 x_0^2x_1^2+a_3 x_0x_1^3+a_4 x_1^4,\\
	s_{01}&= b_0 x_0^3+b_1 x_0^2x_1+b_2 x_0x_1^2+b_3 x_1^3,\\
	s_{02}&= c_0 x_0^3+c_1 x_0^2x_1+c_2 x_0x_1^2+c_3 x_1^3,\\
	s_{11}&= d_0 x_0^2+d_1 x_0x_1+d_2 x_1^2,\\
	s_{12}&= g_0 x_0^2+g_1 x_0x_1+g_2 x_1^2,\\
	s_{22}&= h_0 x_0^2+h_1 x_0x_1+h_2 x_1^2.
\end{align*}
In the ambient $\PP^6$ consider the quadratic homogeneous polynomial
\begin{align*}
	Q\,:=\;& a_0 z_0^2+a_1 z_0z_1+a_2 z_1^2+a_3 z_1z_2+a_4 z_2^2+b_0 z_0z_3+b_1 z_1z_3+b_2 z_2z_3+d_0 z_3^2\\
	&+b_3 z_2z_4+d_1 z_3z_4+d_2 z_4^2+c_0 z_0z_5+c_1 z_1z_5+c_2 z_2z_5+g_0 z_3z_5+g_1 z_4z_5\\
	&+h_0 z_5^2+c_3 z_2z_6+g_2 z_4z_6+h_1 z_5z_6+h_2 z_6^2.
\end{align*}
For $b\neq 0$ the hypersurface $\{Q=0\}$ cuts the scroll $Y_{a,b}\cong\Tcal_{2,1,1}$ along a splitting conic bundle of multidegree $(4,3,3,2,2,2)$

On $Y_{1,0}$ the coordinate $z_6$ is independent (cone vertex). Pulling back by the morphism $g:\Tcal_{2,2,0}\rightarrow Y_{1,0}\subset\PP^6$ in (\ref{morg}) the equation $Q=0$ becomes a splitting conic bundle
\[
F_{1,0}\;=\; s_{00}\,y_0^2+s_{01}\,y_0y_1+s_{02}\,y_0y_2+s_{11}\,y_1^2+s_{12}\,y_1y_2+s_{22}\,y_2^2\;\subset\;\Tcal_{2,2,0},
\]
where
\stepcounter{thm}
\begin{equation}\label{eq-P6-deg-(2,2,0)}
	\begin{aligned}
		s_{00}&= a_0 x_0^4+a_1 x_0^3x_1+a_2 x_0^2x_1^2+a_3 x_0x_1^3+a_4 x_1^4,\\
		s_{01}&= b_0 x_0^4+b_1 x_0^3x_1+b_2 x_0^2x_1^2+b_3 x_0x_1^3\;+\;c_0 x_0^2x_1^2+c_1 x_0x_1^3+c_2 x_1^4,\\
		s_{02}&= c_3 x_1^2,\\
		s_{11}&= d_0 x_0^4+d_1 x_0^3x_1+d_2 x_0^2x_1^2\;+\;g_0 x_0^2x_1^2+g_1 x_0x_1^3+h_0 x_1^4,\\
		s_{12}&= g_2 x_0x_1+h_1 x_1^2,\\
		s_{22}&= h_2,
	\end{aligned}
\end{equation}
i.e. a splitting conic bundle of multidegree $(4,4,2,4,2,0)$ in $\Tcal_{2,2,0}$.

Define the family
\[
\widetilde{\mathcal{C}}\;\coloneqq\;\widetilde{Y}\cap\{Q=0\}\ \subset\ \PP^6\times\PP^1,
\qquad
\mathcal{C}_{a,b}\;\coloneqq\;\widetilde{\mathcal{C}}\cap\big(\PP^6\times\{[a:b]\}\big).
\]
If $h_2\neq 0$, the quadric $Q$ does not pass through the cone vertex $[0:\cdots:0:1]\in Y_{1,0}$, hence $\mathcal{C}_{1,0}$ avoids the vertex and the morphism $\pi:\widetilde{X}\rightarrow\widetilde{Y}$ in Proposition~\ref{prop:pr-birational} induces an isomorphism between $\widetilde{\mathcal{C}}$ and its strict transform in $\widetilde{X}$, that we will denote by $\widetilde{\mathcal{C}}^{X}$, and whose fibers we will denote by $\widetilde{\mathcal{C}}^{X}_{A,B}$. 

Summing up we constructed an inclusion of families $\widetilde{\mathcal{C}}^{X}\subset \widetilde{X}$ parametrized by $\mathbb{P}^1_{[A:B]}$ such that 
$$\widetilde{\mathcal{C}}^{X}_{A,B}\subset X_{A,B}\cong\PP\!\big(\OO_{\PP^1}(2)\oplus\OO_{\PP^1}(1)\oplus\OO_{\PP^1}(1)\big)$$
is a conic bundle of multidegree $(4,3,3,2,2,2)$ for $B\neq 0$, while 
$$\widetilde{\mathcal{C}}^{X}_{1,0}\subset X_{1,0}\cong\PP\!\big(\OO_{\PP^1}(2)\oplus\OO_{\PP^1}(2)\oplus\OO_{\PP^1}\big)$$ 
is a conic bundle of multidegree $(4,4,2,4,2,0)$.

\medskip

Now, consider the codimension five quadric $\mathcal{U} \subset \P(V^1_{4,3,3,2,2,2})$ in Section \ref{sec:433222} parametrizing unirational conic bundles $\mathcal{Q}^1 \subset \P(\mathcal{O}(2) \oplus \mathcal{O}(1) \oplus \mathcal{O}(1))$ with bitangent plane $T_p\mathcal{Q}^1=T_q\mathcal{Q}^1$ at two fixed $k$-points $p,q \in \mathcal{Q}^1(k)$. We have the following:

\begin{thm}\label{unirat-(2,1,1)-(2,2,0)}
	If the conic bundle $\widetilde{\mathcal{C}}^{X}_{0,1} \subset \PP\!\big(\OO_{\PP^1}(2)\oplus\OO_{\PP^1}(1)\oplus\OO_{\PP^1}(1)\big)$ lies in $\mathcal{U} \subset \PP(V^1_{4,3,3,2,2,2})$, then $\widetilde{\mathcal{C}}^{X}_{1,0}\subset\PP\!\big(\OO_{\PP^1}(2)\oplus\OO_{\PP^1}(2)\oplus\OO_{\PP^1}\big)$ is unirational. 
	
	Moreover, the quadric $\mathcal{U}$ specializes to a subvariety $\mathcal{U}'$ of $\PP(V^1_{4,4,2,4,2,0})$ such that the orbit $\chi(\mathcal{G}_{2,2,0} \times \mathcal{U}')$ is Zariski dense in $\PP(V^1_{4,4,2,4,2,0})$.
\end{thm}

\begin{proof}
	The quadric $\mathcal{U}$ parametrizing unirational conic bundles with a bitangent plane is given by the conditions $a_{0}=a_{1}=a_{3}=a_{4}=0$ and $c_{3}=\frac{b_{3}}{b_{0}}c_{0}$. By (\ref{eq-P6-deg-(2,2,0)}) $\widetilde{\mathcal{C}}^{X}_{1,0}$ is a splitting conic bundle of type $(4,4,2,4,2,0)$ with $s_{00}=a_{2}x_{0}^2x_{1}^2$ and $s_{02}=\frac{b_{3}}{b_{0}}c_{0}x_{1}^2$. The curve $C:= \widetilde{\mathcal{C}}^{X}_{1,0} \cap \{y_{1}=0\}$ is given by 
	$$C=\left\lbrace a_{2}x_{0}^2x_{1}^2y_{0}^2+\frac{b_{3}}{b_{0}}c_{0}x_{1}^2y_{0}y_{2}+y_{2}^2=0\right\rbrace \subset \PP(\OO_{\PP^1}(2) \oplus \OO_{\PP^1}).$$
	Setting $x_{0}=1,y_{0}=1$ we get a plane cubic curve 
	$$C'=\left\lbrace a_{2}x_{1}^2+\frac{b_{3}}{b_{0}}c_{0}x_{1}^2y_{2}+y_{2}^2=0\right\rbrace\subset \mathbb{A}^2_{k}$$ with a double point at the origin. Then $C$ yields a rational multisection of $\widetilde{\mathcal{C}}^{X}_{1,0}$, whose unirationality now follows by Proposition \ref{Enr}.
	
	Let $\mathcal{U}'$ be the locus of conic bundles of the form (\ref{eq-P6-deg-(2,2,0)}) and such that $a_{0}=a_{1}=a_{3}=a_{4}=0$ and $c_{3}=\frac{b_{3}}{b_{0}}c_{0}$. Consider now the natural action $$\chi: \mathcal{G}_{2,2,0} \times \PP(V^1_{4,4,2,4,2,0}) \mapsto \PP(V^1_{4,4,2,4,2,0})$$ induced by $(\varphi,\mathcal{Q}^1) \mapsto \varphi^{*}\mathcal{Q}^1$. Restricting to the locus with $h_2 \neq 0$ the differential of $\chi$ at the identity has full rank $21$, proving the last claim.
\end{proof}

\begin{Proposition}
	Let $\mathcal{Q}^1 \in \mathcal{U}'$ be a general conic bundle in $\PP(\OO_{\PP^1}(2) \oplus \OO_{\PP^1}(2) \oplus \OO_{\PP^1})$. Then $\mathcal{Q}^1$ is not rational over $k$.
\end{Proposition}
\begin{proof}
	After a linear change of coordinates we can put $\mathcal{Q}^1$ in the form $ax^2+bx^2=z^2$, with $a,b$ non zero rational function over $\PP^1$, as in Section \ref{Bra}. The conditions on the coefficients of $\mathcal{Q}^1$ imposed by $\mathcal{U}'$ reads
	\begin{align*}
		a(t)&=\frac{a_{2}t^2h_{0}^4}{P_{1}(t)P_{2}(t)}, \\
		b(t)&=\frac{a_{2}t^2h_{0}^2}{P_{2}(t)},
	\end{align*}
	with
	\begingroup\small
	\[
	\begin{aligned}
		P_1(t)&=-b_{0}^{2} t^{8}
		- 2 b_{0} b_{1} t^{7}
		+ 4 a_{2} d_{0} t^{6}
		- 2 b_{0} b_{2} t^{6}
		- b_{1}^{2} t^{6}
		+ 4 a_{2} d_{1} t^{5}
		- 2 b_{0} b_{3} t^{5}\\
		&\quad - 2 b_{1} b_{2} t^{5}
		+ 4 a_{2} d_{2} t^{4}
		- 2 b_{0} b_{4} t^{4}
		- 2 b_{1} b_{3} t^{4}
		- b_{2}^{2} t^{4}
		+ 4 a_{2} d_{3} t^{3}
		- 2 b_{1} b_{4} t^{3}\\
		&\quad - 2 b_{2} b_{3} t^{3}
		+ 4 a_{2} d_{4} t^{2}
		- 2 b_{2} b_{4} t^{2}
		- b_{3}^{2} t^{2}
		- 2 b_{3} b_{4} t
		- b_{4}^{2};
	\end{aligned}
	\]
	\endgroup
	
	\begingroup\small
	\[
	\begin{aligned}
		P_2(t)&=-a_{2} g_{1}^{2} t^{4}
		- a_{2} g_{2}^{2} t^{2}
		- b_{0}^{2} h_{0} t^{8}
		- b_{1}^{2} h_{0} t^{6}
		- b_{2}^{2} h_{0} t^{4}
		- b_{3}^{2} h_{0} t^{2}
		+ b_{4} c_{2} g_{2}
		- c_{2}^{2} d_{0} t^{4}
		- c_{2}^{2} d_{1} t^{3}
		- c_{2}^{2} d_{2} t^{2}
		- c_{2}^{2} d_{3} t
		\\
		&\quad+ 4 a_{2} d_{0} h_{0} t^{6}
		+ 4 a_{2} d_{1} h_{0} t^{5}
		+ 4 a_{2} d_{2} h_{0} t^{4}
		+ 4 a_{2} d_{3} h_{0} t^{3}
		+ 4 a_{2} d_{4} h_{0} t^{2}
		- 2 a_{2} g_{1} g_{2} t^{3}
		\\
		&\quad- 2 b_{0} b_{1} h_{0} t^{7}
		- 2 b_{0} b_{2} h_{0} t^{6}
		- 2 b_{0} b_{3} h_{0} t^{5}
		- 2 b_{1} b_{2} h_{0} t^{5}
		- 2 b_{0} b_{4} h_{0} t^{4}
		- 2 b_{1} b_{3} h_{0} t^{4}
		\\
		&\quad- 2 b_{1} b_{4} h_{0} t^{3}
		- 2 b_{2} b_{3} h_{0} t^{3}
		- 2 b_{2} b_{4} h_{0} t^{2}
		- 2 b_{3} b_{4} h_{0} t
		+ b_{0} c_{2} g_{1} t^{5}
		+ b_{0} c_{2} g_{2} t^{4}
		\\
		&\quad + b_{1} c_{2} g_{1} t^{4}
		+ b_{1} c_{2} g_{2} t^{3}
		+ b_{2} c_{2} g_{1} t^{3}
		+ b_{2} c_{2} g_{2} t^{2}
		+ b_{3} c_{2} g_{1} t^{2}
		+ b_{3} c_{2} g_{2} t
		\\
		&\quad + b_{4} c_{2} g_{1} t
		- c_{2}^{2} d_{4}
		- b_{4}^{2} h_{0}.
	\end{aligned}
	\]
	\endgroup

	The residue at the divisor defined by $P_{2}(t)$ is
	\[
	\frac{h_0^3}{Q_1t^{5}
		+Q_2t^{4}
		+Q_3t^{3}
		+Q_4t^{2}
		+Q_5t
		+Q_6}
	\]
	where 
	\begingroup
	\[
	\begin{aligned}
		Q_1&=b_{0}c_{2}g_{1};\\
		Q_2&=- a_{2}g_{1}^{2};\\
		Q_3&=- 2a_{2}g_{1}g_{2};\\
		Q_4&=- a_{2}g_{2}^{2};\\
		Q_5&=b_{3}b_{4}c_{2}^{2}g_{2}^{2}
		+ b_{4}^{2}c_{2}^{2}g_{1}g_{2}
		- b_{4}c_{2}^{3}d_{3}g_{2};\\
		Q_6&=+ b_{0}c_{2}g_{2}
		+ b_{1}c_{2}g_{1}
		+ b_{1}c_{2}g_{2}
		+ b_{2}c_{2}g_{1}
		+ b_{2}c_{2}g_{2}
		+ b_{3}c_{2}g_{1}- c_{2}^{2}d_{0}
		- c_{2}^{2}d_{1}
		- c_{2}^{2}d_{2}
		- c_{2}^{2}d_{4};
	\end{aligned}
	\]
	
	\endgroup
	
	which is not a square in $k$. Then $\mathcal{Q}^1$ does not admit a rational section and hence is not rational over $k$.
\end{proof}

\subsubsection{Degenerating $\mathcal{T}_{2,2,0}$ to $\mathcal{T}_{3,1,0}$}

For $[a:b]\in\mathbb{P}^1$, consider the threefold
\[
S_{a,b}\;\subset\;\mathbb{P}^6
\]
defined by the $2\times 2$ minors of the $2\times 4$ matrix
\[
M_{a,b}\;=\;
\begin{pmatrix}
	z_0 & z_1 & a\,z_2+b\,z_3 & z_4\\[2pt]
	z_1 & z_2 & a\,z_3+b\,z_4 & z_5
\end{pmatrix}.
\]
Note that $z_6$ does not occur in $M_{a,b}$; thus $S_{a,b}$ is a cone in $\mathbb{P}^6$ with vertex $[0:\cdots:0:1]$ over a surface in the hyperplane $\{z_6=0\}\cong\mathbb{P}^5$.

For $[a:b]\neq[1:0]$ we have $S_{a,b}\ \cong\ \widehat{S(2,2)}$, the cone in $\mathbb{P}^6$ over $S(2,2)\cong\mathbb{P}\!\big(\mathcal{O}_{\mathbb{P}^1}(2)\oplus\mathcal{O}_{\mathbb{P}^1}(2)\big)\subset\mathbb{P}^5$.
At the special point $[a:b]=[1:0]$ we have
\[
M_{1,0}=\begin{pmatrix} z_0&z_1&z_2&z_4\\ z_1&z_2&z_3&z_5\end{pmatrix},
\]
whose $2\times 2$ minors cut the surface scroll $S(3,1)\cong\mathbb{P}(\mathcal{O}(3)\oplus\mathcal{O}(1))\subset\mathbb{P}^5$; thus $S_{1,0}\ \cong\ \widehat{S(3,1)}$ is the cone in $\mathbb{P}^6$ over $S(3,1)$.

In the threefold scroll
\[
\mathcal{T}_{2,2,0}\;=\;\mathbb{P}\!\big(\mathcal{O}_{\mathbb{P}^1}(2)\oplus\mathcal{O}_{\mathbb{P}^1}(2)\oplus\mathcal{O}_{\mathbb{P}^1}\big)
\]
with Cox coordinates $(x_0,x_1;y_0,y_1,y_2)$ consider the splitting conic bundle defined by
\[
F\;=\; s_{00}\,y_0^2+s_{01}\,y_0y_1+s_{02}\,y_0y_2+s_{11}\,y_1^2+s_{12}\,y_1y_2+s_{22}\,y_2^2,
\]
with 
{\setlength{\jot}{1pt}\begin{align*}
		s_{00}&:= a_0x_0^4+a_1x_0^3x_1+a_2x_0^2x_1^2+a_3x_0x_1^3+a_4x_1^4,\\
		s_{01}&:= b_0x_0^4+b_1x_0^3x_1+b_2x_0^2x_1^2+b_3x_0x_1^3+b_4x_1^4,\\
		s_{02}&:= c_0x_0^2+c_1x_0x_1+c_2x_1^2,\\
		s_{11}&:= d_0x_0^4+d_1x_0^3x_1+d_2x_0^2x_1^2+d_3x_0x_1^3+d_4x_1^4,\\
		s_{12}&:= g_0x_0^2+g_1x_0x_1+g_2x_1^2,\\
		s_{22}&:= h_0.
\end{align*}}
In the ambient $\mathbb{P}^6$ consider the quadratic polynomial
\begin{align*}
	G\;:=\;& a_0z_0^2+a_1z_0z_1+a_2z_1^2+a_3z_1z_2+a_4z_2^2
	+ b_0z_0z_3+b_1z_1z_3+b_2z_2z_3+b_3z_2z_4+b_4z_2z_5\\
	&+ c_0z_0z_6+c_1z_1z_6+c_2z_2z_6
	+ d_0z_3^2+d_1z_3z_4+d_2z_4^2+d_3z_4z_5+d_4z_5^2\\
	&+ g_0z_3z_6+g_1z_4z_6+g_2z_5z_6
	+ h_0z_6^2.
\end{align*}
For any $[a:b]$, the hypersurface $\{G=0\}$ restricts to a surface in $S_{a,b}$. When $[a:b]\neq[1:0]$ (so $S_{a,b}\cong \widehat{S(2,2)}$), this surface corresponds, in the bundle coordinates $(x_0,x_1;y_0,y_1,y_2)$ of $\mathcal{T}_{2,2,0}$, to the conic bundle $F=0$ above.

At $[a:b]=[1:0]$ the same polynomial $G$ cuts a surface in $S_{1,0}\cong\widehat{S(3,1)}$. Consider the morphism
\[
g:\ \mathcal{T}_{3,1,0}=\mathbb{P}\!\big(\mathcal{O}_{\mathbb{P}^1}(3)\oplus\mathcal{O}_{\mathbb{P}^1}(1)\oplus\mathcal{O}_{\mathbb{P}^1}\big)\ \longrightarrow\ S_{1,0}\ \subset\ \mathbb{P}^6
\]
given in Cox coordinates by
\[
g([x_0:x_1;y_0:y_1:y_2])\;=\;[\,x_0^3y_0:\ x_0^2x_1y_0:\ x_0x_1^2y_0:\ x_1^3y_0:\ x_0y_1:\ x_1y_1:\ y_2\,].
\]
Then $g^*(G)=0$ is a splitting conic bundle in $\mathcal{T}_{3,1,0}$ defined by
\[
F_{1,0}\;=\; s_{00}\,y_0^2+s_{01}\,y_0y_1+s_{02}\,y_0y_2+s_{11}\,y_1^2+s_{12}\,y_1y_2+s_{22}\,y_2^2,
\]
with coefficients
{\setlength{\jot}{1pt}
\stepcounter{thm}
	\begin{equation}\label{deg-(2,2,0-(3,1,0)}
		\begin{aligned}
			s_{00}&:= a_0x_0^6+a_1x_0^5x_1+a_2x_0^4x_1^2+a_3x_0^3x_1^3+a_4x_0^2x_1^4
			+ b_0x_0^3x_1^3+b_1x_0^2x_1^4+b_2x_0x_1^5+d_0x_1^6,\\
			s_{01}&:= b_3x_0^2x_1^2+b_4x_0x_1^3+d_1x_0x_1^3,\\
			s_{02}&:= c_0x_0^3+c_1x_0^2x_1+c_2x_0x_1^2+g_0x_1^3,\\
			s_{11}&:= d_2x_0^2+d_3x_0x_1+d_4x_1^2,\\
			s_{12}&:= g_1x_0+g_2x_1,\\
			s_{22}&:= h_0.
		\end{aligned}
	\end{equation}
}

Let $\mathbb{P}_{a,b}(2,2,0)$ denote the family obtained by identifying $S_{a,b}\cong\widehat{S(2,2)}$ for $[a:b]\neq[1:0]$ and specializing to $\widehat{S(3,1)}$ for $b\to 0$. Inside it, let
\[
\mathcal{C}_{a,b}(4,4,0)\;\subset\;\mathbb{P}_{a,b}(2,2,0)
\]
be the family of splitting conic bundles cut by the single quadratic equation $G=0$ on $S_{a,b}$. Then, as $b\to 0$ the family $\mathcal{C}_{a,b}(4,4,0)$ specialized to a splitting conic bundle of multidegree $(6,4,3,2,1,0)$ in $\widehat{S(3,1)}$. Finally, to construct a family whose general and special fibers are actually isomorphic to $\mathbb{P}(2,2,0)$ and $\mathbb{P}(3,1,0)$ respectively, one embeds the projective bundles in $\mathbb{P}^9$, and takes the pull back of the family we constructed via a suitable relative projection as in Section \ref{(2,1,1)->(2,2,0)}.

\begin{thm}
	Let $\mathcal{U} \subset \PP(V^1_{4,4,2,4,2,0})$ be the linear subspace given by $a_{3}=a_{4}=c_{2}=0$. A general $\mathcal{C}_{0,1} \in \mathcal{U}$ is unirational and its specialization $\mathcal{C}_{1,0} \subset \PP(\OO_{\PP^1}(3) \oplus \OO_{\PP^1}(1) \oplus \OO_{\PP^1})$ is unirational as well. 
\end{thm}
\begin{proof}
	The first claim follows from Propositions \ref{prop:TTUni442420} and \ref{prop:dom442420}. For the second claim note that by (\ref{deg-(2,2,0-(3,1,0)}) $s_{11},s_{12},s_{22}$ do not involve the coefficients $a_{3},a_{4},c_{2}$. Since by Proposition \ref{prop:case643210} a conic bundle $\mathcal{Q}^1 \subset \PP(\OO_{\PP^1}(3) \oplus \OO_{\PP^1}(1) \oplus \OO_{\PP^1})$ admits a rational multisection as long as $s_{12}^2-s_{11}s_{22} \neq 0$, we have that $\mathcal{C}_{1,0} \subset \PP(\OO_{\PP^1}(3) \oplus \OO_{\PP^1}(1) \oplus \OO_{\PP^1})$ is unirational.
\end{proof}

\subsubsection{Degenerating $\mathcal{T}_{3,1,0}$ to $\mathcal{T}_{4,0,0}$}

For $[a:b]\in\mathbb{P}^1$, consider the threefold
\[
S_{a,b}\;\subset\;\mathbb{P}^6
\]
defined by the $2\times 2$ minors of the $2\times 4$ matrix
\[
M_{a,b}\;=\;
\begin{pmatrix}
	z_0 & z_1 & z_2 & a\,z_3+b\,z_4\\[2pt]
	z_1 & z_2 & z_3 & a\,z_4+b\,z_5
\end{pmatrix}.
\]
Note that $z_6$ does not appear in $M_{a,b}$, so $S_{a,b}$ is a cone inside $\mathbb{P}^6$.

If $[a:b]\neq[1:0]$ we have that $S_{a,b}\ \cong\ \widehat{S(3,1)}$ is a cone in $\mathbb{P}^6$ over the scroll$S(3,1)\subset\mathbb{P}^5$ with vertex the point $[0:\cdots:0:1]$. At the special point $[a:b]=[1:0]$ we have
\[
M_{1,0}=\begin{pmatrix} z_0 & z_1 & z_2 & z_3\\ z_1 & z_2 & z_3 & z_4 \end{pmatrix},
\]
so the $2\times 2$ minors involve only $z_0,\dots,z_4$. Therefore $S_{1,0}$ is a cone whose vertex is the line
\[
L=\{z_0=\cdots=z_4=0\}\ \cong\ \mathbb{P}^1\subset\mathbb{P}^6,
\]
over the rational normal quartic curve in $\mathbb{P}^4=\{z_5=z_6=0\}$ given by the minors of $M_{1,0}$.

On the threefold scroll
\[
\mathcal{T}_{3,1,0}=\mathbb{P}\!\big(\mathcal{O}_{\mathbb{P}^1}(3)\oplus\mathcal{O}_{\mathbb{P}^1}(1)\oplus\mathcal{O}_{\mathbb{P}^1}\big)
\]
use Cox coordinates $(x_0,x_1;y_0,y_1,y_2)$ and consider the splitting conic bundle defined by
\[
F\;=\; s_{00}\,y_0^2+s_{01}\,y_0y_1+s_{02}\,y_0y_2+s_{11}\,y_1^2+s_{12}\,y_1y_2+s_{22}\,y_2^2,
\]
with
{\setlength{\jot}{1pt}\begin{align*}
		s_{00}&:= a_0x_0^6+a_1x_0^5x_1+a_2x_0^4x_1^2+a_3x_0^3x_1^3+a_4x_0^2x_1^4+a_5x_0x_1^5+a_6x_1^6,\\
		s_{01}&:= b_0x_0^4+b_1x_0^3x_1+b_2x_0^2x_1^2+b_3x_0x_1^3+b_4x_1^4,\\
		s_{02}&:= c_0x_0^3+c_1x_0^2x_1+c_2x_0x_1^2+c_3x_1^3,\\
		s_{11}&:= d_0x_0^2+d_1x_0x_1+d_2x_1^2,\\
		s_{12}&:= g_0x_0+g_1x_1,\\
		s_{22}&:= h_0.
\end{align*}}
In the ambient $\mathbb{P}^6$ consider the quadratic polynomial
\begin{align*}
	G\;:=\;& a_0z_0^2+a_1z_0z_1+a_2z_1^2+a_3z_1z_2+a_4z_2^2+a_5z_2z_3+a_6z_3^2 \\
	&+ b_0z_0z_4+b_1z_1z_4+b_2z_2z_4+b_3z_3z_4+b_4z_3z_5 \\
	&+ c_0z_0z_6+c_1z_1z_6+c_2z_2z_6+c_3z_3z_6 \\
	&+ d_0z_4^2+d_1z_4z_5+d_2z_5^2 \\
	&+ g_0z_4z_6+g_1z_5z_6 + h_0z_6^2.
\end{align*}
For every $[a:b]$, the hypersurface $\{G=0\}$ intersects $S_{a,b}$ in a surface. When $[a:b]\neq[1:0]$ in the bundle coordinates of $\mathcal{T}_{3,1,0}$ this surface corresponds to the conic bundle $F=0$ above.

At $[a:b]=[1:0]$ the same polynomial $G$ cuts in $S_{1,0}$ a surface. Consider the morphism
\[
g:\ \mathcal{T}_{4,0,0}=\mathbb{P}\!\big(\mathcal{O}_{\mathbb{P}^1}(4)\oplus\mathcal{O}_{\mathbb{P}^1}\oplus\mathcal{O}_{\mathbb{P}^1}\big)\ \longrightarrow\ S_{1,0}\ \subset\ \mathbb{P}^6
\]
given in Cox coordinates by
\[
g([x_0:x_1;y_0:y_1:y_2])\;=\;[\,x_0^4y_0:\ x_0^3x_1y_0:\ x_0^2x_1^2y_0:\ x_0x_1^3y_0:\ x_1^4y_0:\ y_1:\ y_2\,].
\]
Then $g^*(G)=0$ is a splitting conic bundle in $\mathcal{T}_{4,0,0}$, defined by
\[
F_{1,0}\;=\; s_{00}\,y_0^2+s_{01}\,y_0y_1+s_{02}\,y_0y_2+s_{11}\,y_1^2+s_{12}\,y_1y_2+s_{22}\,y_2^2,
\]
with 
{\setlength{\jot}{1pt}
\stepcounter{thm}
	\begin{equation}\label{deg-(3,1,0)-(4,0,0)}
		\begin{aligned}
			s_{00}&:= a_0x_0^8+a_1x_0^7x_1+a_2x_0^6x_1^2+a_3x_0^5x_1^3+a_4x_0^4x_1^4+a_5x_0^3x_1^5+a_6x_0^2x_1^6\\
			& + b_0x_0^4x_1^4+b_1x_0^3x_1^5+b_2x_0^2x_1^6+b_3x_0x_1^7+d_0x_1^8,\\
			s_{01}&:= b_4x_0x_1^3+d_1x_1^4,\\
			s_{02}&:= c_0x_0^4+c_1x_0^3x_1+c_2x_0^2x_1^2+c_3x_0x_1^3+g_0x_1^4,\\
			s_{11}&:= d_2,\\
			s_{12}&:= g_1,\\
			s_{22}&:= h_0.
		\end{aligned}
\end{equation}}
Let $\mathbb{P}_{a,b}(3,1,0)$ denote the family obtained by identifying $S_{a,b}\cong\widehat{S(3,1)}$ for $[a:b]\neq[1:0]$ and specializing, as $b\to 0$, to the cone with line vertex over the rational normal quartic. Let
\[
\mathcal{C}_{a,b}(3,1,0)\;\subset\;\mathbb{P}_{a,b}(3,1,0)
\]
be the corresponding family of splitting conic bundles cut by the quadratic equation $G=0$ on $S_{a,b}$. Then, as $b\to 0$, $\mathcal{C}_{a,b}(3,1,0)$ specialized to a splitting conic bundle of multidegree $(8,4,4,0,0,0)$ in $\mathcal{T}_{4,0,0}$. Finally, to construct a family whose general and special fibers are actually isomorphic to $\mathbb{P}(3,1,0)$ and $\mathbb{P}(4,0,0)$ respectively, one embeds the projective bundles in $\mathbb{P}^9$, and takes the pull back of the family we constructed via a suitable relative projection as in Section \ref{(2,1,1)->(2,2,0)}.

\begin{thm}\label{(3,1,0-(4,0,0)}
	Let $\mathcal{C}_{0,1} \subset \PP(\OO_{\PP^1}(3) \oplus \OO_{\PP^1}(1) \oplus \OO_{\PP^1})$ be a unirational conic bundle with $s_{12}^2-s_{11}s_{22} \neq 0$ and $(a_0h_0-c_0^2)d_0-g_0^2a_0 \in (k^*)^2$. Then the conic bundle $\mathcal{C}_{1,0} \subset \PP(\OO_{\PP^1}(4) \oplus \OO_{\PP^1} \oplus \OO_{\PP^1})$ is also unirational.
\end{thm}
\begin{proof}
	By Proposition \ref{prop:Me844000}, in the notation of (\ref{deg-(3,1,0)-(4,0,0)}), a conic bundle $$\mathcal{C}_{1,0} \subset \PP(\OO_{\PP^1}(4) \oplus \OO_{\PP^1} \oplus \OO_{\PP^1})$$ is unirational when $h_0 \in k^*,g_1^2-d_2h_0 \in k^*$ and the leading coefficient of $-\frac{4}{h_0}\delta_{\mathcal{C}_{1,0}}$, written as a polynomial in $t=\frac{x_0}{x_1}$ is a square. The first two conditions are satisfied thank to the generality assumption. For the last condition we expand $\delta_{\mathcal{C}_{1,0}}$ as a polynomial in $P_{\delta} \in k(a_0,a_1,\dots,h_0)[t]$ and check that its leading coefficient, up to a linear change of variables, is $(a_0h_0-c_0^2)d_0-g_0^2a_0$ which is a square by assumption. 
\end{proof}

\begin{Remark}
	Note that using a similar argument of the proof of Lemma \ref{lem:d8} we have that the locus of conic bundles $\mathcal{Q}^1 \subset \PP(\OO_{\PP^1}(3) \oplus \OO_{\PP^1}(1) \oplus \OO_{\PP^1})$ that satisfy the two conditions of Theorem \ref{(3,1,0-(4,0,0)} is Zariski dense in $\PP(V^1_{6,4,3,2,1,0})$, but in general not Zariski open.
\end{Remark}
\section{Higher degree discriminants}\label{sec:beyond}
It is natural to ask whether conic bundles with more than eight singular fibers can be unirational. Constructions similar to those above work in further cases; however, the number of components of the parameter space grows quadratically.

\begin{Proposition}\label{lem:countDn}
Let $P_n$ be the number of components of the parameter space of conic bundles with $n$ singular fibers. Then
\[
\sum_{n\ge 0} P_n\,q^n=\frac{1}{(1-q^2)(1-q^3)(1-q^4)}.
\]
In particular, $P_n$ grows quadratically with $n$.
\end{Proposition}

\begin{proof}
Following \eqref{compdeg}, the number $P_n$ is found as follows: choose non negative integers $d_{0,0},d_{1,1},d_{2,2}$ with sum $n$, and integers $a_0,a_1,a_2$ such that
$$
d_{0,0}-2a_0=d_{1,1}-2a_1=d_{2,2}-2a_2=:t.
$$
Then $d_{0,1}=t+a_0+a_1$, $d_{0,2}=t+a_0+a_2$, $d_{1,2}=t+a_1+a_2$. The triple $(a_0,a_1,a_2)$ is determined up to $(a_0,a_1,a_2)\sim(a_0+k,a_1+k,a_2+k)$, so each class has a unique representative with $a_2=0$, and the number of equivalence classes is $P_n$. Given $(d_{0,0},d_{1,1},d_{2,2})$, its representative is
$$
\left(\tfrac{d_{0,0}-d_{1,1}}{2},\tfrac{d_{0,0}-d_{2,2}}{2},0\right)
$$
which is integral if and only if $d_{0,0},d_{1,1},d_{2,2}$ have the same parity. Thus $P_n$ counts weak partitions of $n$ into three parts of the same parity; the generating function is as stated (\href{https://oeis.org/A005044}{OEIS A005044}). The quadratic growth follows from the explicit formula (found in \emph{ibid})
$$P_n=\mathrm{rd}\left( \frac{n^2}{12}\right)-\left\lfloor \frac{n}{4}\right\rfloor \left\lfloor \frac{n+2}{4}\right\rfloor$$ 
where $\mathrm{rd}$ denotes rounding, and $\lfloor x \rfloor$ denotes the floor of $x$.
\end{proof}

\begin{Remark}
The sequence $P_n$ is known as Alcuin's sequence, and appears when generalizing a problem mentioned by the Carolingian monk Alcuin of York \cite{Alcuin2}.
\end{Remark}

\section{Computer algebra scripts}
\subsection*{Maple scripts}
Companion Maple worksheets implementing the constructions in this paper are available at
\begin{center}
\url{https://github.com/msslxa/Degree-8-surface-conic-bundles}
\end{center}
They build the coefficient polynomials of the conic bundles over $\mathbb{P}^1$, diagonalize the ternary quadratic via the explicit change of variables in Proposition~\ref{diag}, and output the associated Brauer model
$a(t)x^2+b(t)y^2-z^2=0$ together with the relevant residues. The repository includes scripts for the four multidegree types treated in the text: $(4,0,0)$, $(3,1,0)$, $(2,2,0)$, and $(2,1,1)$. 

\subsection*{Magma verification for type $(4,0,0)$}
Alongside the Maple worksheets, the repository also contains the file
\texttt{(8,4,4,0,0,0)\_Cremonas.txt}, a \textsc{Magma} script that verifies the Cremona transformation-based proof in Proposition~\ref{prop:U12}.
The script works over the rational function field
$$
\mathbb{Q}(a_4,a_5,a_6,b_2,b_3,c_0,c_1,c_2,c_3,c_4,d_0,g_0,h_0)
$$
imposes the linear relations defining $\Ucal^{12}$, embeds the surface $X_8\subset\PP^3$ with the $6$-fold line,
and constructs the plane octic $C=X_8\cap(z_1-z_2=0)$ with ordinary double points at the prescribed three points and multiplicity $6$ at $q$. It then applies consecutively:
\begin{itemize}
\item[-] the quadratic Cremona transformation $\phi_1$ with base points $c,n_1,n_3$ on $H\simeq\PP^2$ (sending $C$ to a sextic $C_1$);
\item[-] the quadratic Cremona transformation $\phi_2$ with base points $q_1,q_2,q_3$ (sending $C_1$ to a quartic $C_2$);
\item[-] the explicit quadratic map $\phi_3:[w_0:w_1:w_2]\mapsto[-2(w_0-w_1)w_2+w_0^2:\;w_1^2:\;w_0w_1]$,
which sends $C_2$ to a conic $C_3$.
\end{itemize}
The script outputs the defining equations at each step and confirms that $C_3$ is a (smooth) conic, hence rational; consequently $C$ is rational, giving the unirational multisection used in the proof.

\subsection*{Magma script computing the cohomology of the normal bundle}
The script computes the numbers $h^0(N_{S/\Tcal_{a_0,a_1,a_2}})$ and $h^1(N_{S/\Tcal_{a_0,a_1,a_2}})$ where $S$ is a random conic bundle of type $(2,1,1)$, $(2,2,0)$, $(3,1,0)$, $(4,0,0)$.

\bibliographystyle{amsalpha}
\bibliography{Biblio}
\end{document}